\documentclass[reqno]{amsart}

\usepackage[margin=2cm]{geometry}
\usepackage{multicol}
\usepackage{xcolor}
\usepackage{hyperref}
\usepackage{mathtools}
\usepackage{array}
\hypersetup{linkbordercolor={green}}
\newcommand{\draft}{}

\usepackage{ifthen}       
\usepackage{xparse}       

\usepackage{scrtime} 
\usepackage{multirow}

\newcommand{\markerO}{\fbox{\rule{0pt}{0.1ex}\textbf{\textsf{Olaf}}}}
\newcommand{\markerF}{\fbox{\rule{0pt}{0.1ex}\textbf{Fernando}}}
\newcommand{\markerJ}{\fbox{\rule{0pt}{0.1ex}\textbf{John}}}
\newcommand{\look}[1]{\markerO \textbf{*}
    \footnote{ #1 }}
\newcommand{\lookO}[1]{\markerO\textbf{*}
    \footnote{\textbf{\textsf{Olaf:}} #1 }}
\newcommand{\lookF}[1]{\markerF\textbf{*}
    \footnote{\textbf{Fernando:} #1 }}
\newcommand{\lookJ}[1]{\markerJ\textbf{*}
	\footnote{\textbf{John:} #1 }}

\newcommand{\Ignore}[1]{}

\ifthenelse{\isundefined \draft }
{
  \renewcommand{\look}[1]{}
  \renewcommand{\lookO}[1]{}%
  \renewcommand{\lookF}[1]{}%
  \renewcommand{\lookJ}[1]{}%

  \ifthenelse{\isundefined \AMSstylefile } 
  {\automark[subsection]{section}}  %
  {}
}
\usepackage{graphicx}                

\usepackage{amsmath, amssymb}
\usepackage{amsthm}
\usepackage{amscd}

\usepackage{bbm}         

\usepackage{float}

\numberwithin{equation}{section}

\newcounter{myenumi}

  \ifthenelse{\isundefined \OlafsVersion} 
{ 
  \theoremstyle{plain} 
  \newtheorem{theorem}{Theorem}[section]

  \newtheorem{maintheorem}[theorem]{Main Theorem}
  \newtheorem{proposition}[theorem]{Proposition}
  \newtheorem{lemma}[theorem]{Lemma}
  \newtheorem{corollary}[theorem]{Corollary}
  \newtheorem{conjecture}[theorem]{Conjecture}
  
  \theoremstyle{definition}       
  \newtheorem{definition}[theorem]{Definition}
  \newtheorem{assumption}[theorem]{Assumption}
  \newtheorem{example}[theorem]{Example}
  
  \newtheorem{remark}[theorem]{Remark}
  \newtheorem*{remark*}{Remark}
  \newtheorem{notation}[theorem]{Notation}

} 
{ 

  \newcommand{\myfont}{\sffamily}

  \newtheoremstyle{mythmstyle}
  {\topsep}
  {\topsep}
  {\itshape}
  {}
  {\bfseries \myfont}
  {.}
  {.5em}
  {}
  
  \newtheoremstyle{mydefstyle}
  {\topsep}
  {\topsep}
  {\normalfont}
  {}
  {\bfseries \myfont}
  {.}
  {.5em}
  {}

  \swapnumbers                    

  \theoremstyle{mythmstyle}       

  \newtheorem{theorem}{Theorem}[section]
  
  \newtheorem{proposition}[theorem]{Proposition}
  \newtheorem{lemma}[theorem]{Lemma}
  \newtheorem{corollary}[theorem]{Corollary}
  
  \newcounter{intro}

  \theoremstyle{mydefstyle}        
  \newtheorem{definition}[theorem]{Definition}
  
  \newtheorem{example}[theorem]{Example}

  \newtheorem{remark}[theorem]{Remark}
  
  \newtheorem*{remark*}{Remark}
  
  \expandafter\let\expandafter\oldproof\csname\string\proof\endcsname
  \let\oldendproof\endproof
  \renewenvironment{proof}[1][\bfseries\myfont\proofname]{%
    \oldproof[\bfseries \myfont #1]%
  }{\oldendproof}
} 

\ifthenelse{\isundefined \OlafsVersion}
{} 
{ 
  \input{better-ams-style.tex}
} 

\DeclareMathSymbol{\shortminus}{\mathbin}{AMSa}{"39}

\newcommand{\compl}[1]{#1^{\mathsf c}}

\renewcommand{\phi}{\varphi}   

\newcommand{\R}{\mathbb{R}} 
\newcommand{\N}{\mathbb{N}} 

\newcommand{\1}{\mathbbm 1}                    

\newcommand{\cC}{{\mathcal C}}

\newcommand{\cL}{{\mathcal L}}

\newcommand{\cP}{{\mathcal P}}

\newcommand{\cR}{{\mathcal R}}

\newcommand{\quadtext}[1]{\quad\text{#1}\quad}

\DeclareMathAlphabet{\Ma}{U}{msa}{m}{n}
\DeclareMathAlphabet{\Mb}{U}{msb}{m}{n}
\DeclareMathAlphabet{\Meuf}{U}{euf}{m}{n}
\DeclareSymbolFont{ASMa}{U}{msa}{m}{n}
\DeclareSymbolFont{ASMb}{U}{msb}{m}{n}

\newcommand{\lessWithNumber}[1]{\stackrel{#1}\preccurlyeq}

\NewDocumentCommand{\less}{o}{%
  \IfNoValueTF{#1}
    {\preccurlyeq}
    {\lessWithNumber{#1}}%
}

\usepackage{tikz}
\usetikzlibrary{arrows.meta}
\usetikzlibrary{calc}
\usetikzlibrary{positioning, fit}
\tikzset{
    position/.style args={#1:#2 from #3}{
        at=(#3.#1), shift=(#1:#2)
    }
}
\usetikzlibrary{arrows.meta,trees,positioning,arrows,fit,shapes}

\usepackage{verbatim} 
\usepackage{subfig}
\usepackage{cleveref}

\title{Geometric and spectral analysis on weighted digraphs}

\author{Fernando Lled\'o and Ignacio Sevillano} %
\address{Department of Mathematics, University Carlos III de Madrid,
  Avda. de la Universidad 30, 28911. Legan\'es (Madrid), Spain and
  Instituto de Ciencias Matem\'aticas (CSIC-UAM-UC3M-UCM), Madrid}
\email{flledo@math.uc3m.es}
\email{nachosemu@gmail.com}

\thanks{
The first named author was supported by Spanish Ministry of Economy and Competitiveness through project DGI MTM2017-84098-P, from the
Severo Ochoa Programme for Centres of Excellence in R{\&}D (SEV-2015-0554) and from the Spanish National Research Council, through the 
\textit{Ayuda extraordinaria a Centros de Excelencia Severo Ochoa} (20205CEX001) and grant
6G-INTEGRATION-3 (TSI-063000-2021-127), funded by UNICO program (under the Next Generation EU umbrella funds), Ministerio de Asuntos Económicos y Transición Digital of Spain.
}


\begin{document}


\ifthenelse{\isundefined \draft}
{\date{\today}}  
{\date{\today. }} 

\begin{abstract} 
In this article we give a geometrical description of the (in general non-selfadjoint) in/out Laplacian $\cL^{+/-}=(d^{+/-})^*d$ and adjacency matrix on digraphs with arbitrary weights, where $(d^{+/-})^*$ is the adjoint of the evaluation map $d^{+/-}$ on the terminal/initial vertex of each arc and $d=d^{+}+d^-$ denotes the discrete gradient.
We prove that the multiplicity of the zero eigenvalue of $\cL^{+/-}=(d^{+/-})^*d$ coincides with the number of sources/sinks of the digraph. We also show that for an acyclic digraph with combinatorial weights the spectrum is contained in the set of non-zero integers. The geometrical perspective allows to interpret the set of circulations $\cC$ of a weighted digraph as coclosed forms on the arcs, i.e. as the kernel of the discrete divergence $d^*$. Moreover, $\cC$ is perpendicular to the set of discrete gradients of functions on the vertices. We also give formulas to compute the capacity of a cut and the value of a flow in terms of $\cL^-$ and $d$. 
We illustrate the results with many concrete examples.
\end{abstract}

\maketitle


%
\section{Introduction}
\label{sec:intro}
%

Spectral graph theory for undirected graphs has been widely studied in the last decades, also in applications to different fields like analysis, combinatorics, discrete geometry or computer science (see, e.g., \cite{chung97,cvetkovic:95,spielman:12}). It consists of the interplay between the spectra of different operators defined on the graph (typically the adjacency or Laplacian matrices) and the geometry, the topology, or the combinatorics of the underlying graph. Spectral graph theory for directed graphs (or digraphs for short) requires often a refined analysis of notions introduced for undirected graphs 
(see \cite{b-brualdi:92} and \cite{marijuan:18} for the relation between connectivity and strong connectivity). Other examples are the 
concept of chromatic number in \cite{neumann-lara:82,mohar:10}), the notion of 
complementarity eigenvalues in \cite{trevisan:23}
or the refinement of Cheeger's constant in the context of digraphs (see, for example, \cite{chung:05} or 
\cite{balti:17,anne-et-al:19} where, in the later reference, also unbounded operators on infinite graphs are considered). 
One of the main differences between (undirected) graphs and digraphs in this framework is that, in general, the adjacency matrices and Laplacian matrices are not selfadjoint for digraphs as opposed to undirected graphs. This leads to the inconvenience that, in general, the corresponding spectra are not real for digraphs. 
Several authors have considered symmetric versions of adjacency or Laplacian matrices on digraphs and relate certain properties of the digraph to the corresponding real spectra 
(see, for example, \cite{chung:05,guo-mohar:17,sander:20,sahoo:21}). 
Alternatively, one can analyze the complex spectrum of non-Hermitian Laplacian or adjacency matrices (cf., 
\cite{godsil:82,dalfo:17}). In this case one can localize the spectrum in certain convex regions of the complex plane \cite{agaev:05,anne-et-al:19,bauer:12}.
In \cite{gnang-murphy:20} the authors generalize variational estimates in the Rayleigh quotients quantifying the discrepancy between Hermitian and non-Hermitian
matrices. Numerical and other applications of directed Laplacians on networks are considered in 
\cite{boley:21,furutani:20,bjorner-lovasz:92,veerman:19}.
Finally, we refer to the classical article \cite{brualdi:10} for a review mainly devoted to the oriented adjacency matrix or the recent books covering relevant topics on digraphs \cite{bang-jensen-gutin:18,bondy-murty:08}; for applications see also \cite{bang-jensen-gutin:10,marijuan:16} and references therein.

In this article we give first a geometrical description of Laplacians and adjacency operators on digraphs with arbitrary weights. Given a digraph $G=(V,A,\partial)$ with incidence $\partial$ and weights on vertices $V$ and arcs $A$ denoted by $m$
we decompose the discrete gradient $d$ and the weighted discrete divergence $d^*$ in terms of the evaluation maps (and their adjoints)
\[
 d^{\pm}\colon \ell_2(V,m)\to \ell_2(A,m)\quadtext{with} 
               \left(d^{\pm}\right)\!(a)=\pm \varphi\!\left(\partial_{\pm}^{} a\right)\;,
\]
where $\partial_{-/+} a$ denotes the initial/terminal vertex of the arc $a\in A$ 
(see also Section~\ref{sec:graph.theory} for additional details and motivation).
The discrete gradient $d$ decomposes as $d=d^+ + d^-$ and the weighted Laplacian is
expressed as the second order operator
$\cL=d^*d$ which can be written as $\cL=\cL^{+}+\cL^{-}$, where the in/out Laplacians are defined by 
\[
 \cL^{+/-}:=\left(d^{+/-}\right)^* d\;.
\]
Similar formulas also hold for the adjacency operator (cf.~Section~\ref{sec:graph.theory}). Note that, by construction, the in/out Laplacians are not selfadjoint and that the usual discrete Laplacian for arbitrary weights (even considering a magnetic field, see, for example, \cite[Section~3]{fclp:22a}) does not see the underlying orientation of the graph. We refer to \cite{lim:20} for a higher order version of the Laplacian on graphs.

Writing the discrete gradient, the divergence or the Laplacians as abstract operators (the corresponding matrices being just {\em coordinates} of the operators once a numbering of vertices and edges are chosen) has many advantages besides its elegance. It allows, for example, to show that certain spectral properties of the graph are intrinsic to the operator and independent of the weights. In many applications, the use of weights is an important ingredient (e.g., capacities in transportation networks or random walks). Moreover, the geometrical point of view emphasized in this article connects naturally with the language of partial differential equations, generalizes easily to the context 
of infinite graphs and facilitates to show spectral characterizations of graph properties independently of the chosen weights. In fact, most of the references mentioned in the first paragraph consider, without stating it explicitly, combinatorial or normalized weights (in some cases edge weights are also used). In addition, natural connections to so-called quantum graphs can be established by this approach and using so-called normalized weights. 
Recall that a metric (or quantum) graph is an intermediate structure between discrete graphs and manifolds (see, e.g., \cite{cattaneo:97,kuchment:08,lledo-post:08}). 
A quantum graph is defined as a metric graph together with a 
selfadjoint operator modeling the dynamics of the system. A typical example is a Laplace-
like operator acting on each edge as the second derivative with certain conditions on the
function value and its derivative at each vertex, turning it into a selfadjoint operator.
We also mention some results that relate central geometrical notions like isospectrality or Cheeger's inequalities in the context of graphs to illustrate the deep connection between graphs and manifolds \cite{alon:85,brooks:99,llpp:15}. 

In this article we give a spectral characterization of sinks and sources of weighted digraphs. We also present a geometrical generalization of circulations and flows on weighted digraphs and give a geometrical interpretation of these combinatorial quantities as coclosed forms on arcs. We show that circulations (which have a natural {\em Kirchhoff-type conservation law} at each vertex) are orthogonal to the discrete gradient. We generalize this result to flows on digraphs, where the conservation law now applies at all vertices except to a distinguished set of vertices $W$, which can be interpreted as the union of sources and sinks of a transportation network. In this case, the flow turns to be orthogonal to the discrete gradient applied to functions with Dirichlet conditions at $W$.
Finally, we give new expressions for the capacity of a cut and the value of a flow in weighted digraphs in terms of the out Laplacian and the discrete derivative and divergence of the graph. For this, it is useful to reinterpret the capacity of the digraph as a weight on the arcs and relate the capacity of a cut with the corresponding capacity-weighted out-Laplacian. We illustrate through the whole article the definitions and results with many concrete examples.\\[1mm]

\paragraph{\bf Structure of the article and main results:} In the next section, we introduce the necessary notions and results for weighted digraphs needed later. In particular, we define and give a natural geometrical description of in/out components of the weighted Laplacian and adjacency operators.
In Section~\ref{sec:connectedness} we introduce relevant notions on a weighted digraph, including the notions of sources and sinks. 
In Proposition~\ref{prop:compression_on_SF} and
\ref{prop:spec_par} we analyze the compressions of the in/out Laplacians to sources/sinks and describe a decomposition of their spectra. The main result in this section is a general spectral characterization of sources and sinks of a weighted digraph. We show in  Theorem~\ref{theo:zero_io_lap} that the multiplicity of the zero
eigenvalue of $\cL^{+/-}$ coincides with the number of sources/sinks of the weighted digraph. We prove that for an acyclic digraph with combinatorial weights the spectrum is contained in the set of non-zero integers. The relation $u \rightsquigarrow v$ (the  vertex $v$ is reachable from $v$ with a directed path) defines a preorder in the vertex set $V$ and we consider chains and maximal chains in the study of properties of digraphs.

In Section~\ref{sec:flows} we apply the previous geometrical analysis to networks. 
The set of circulations $\cC$ of a network, i.e., functions on the arcs with a natural Kirchhoff type condition at any vertex, is perpendicular to the
discrete gradient of functions on the vertices {(\em exact one forms)} and
can be described geometrically as {\em coclosed forms on arcs}, i.e., as the kernel of the discrete divergence:
\[
 \cC=\mathrm{ker}d^* \quadtext{and} \cC \perp \{d\varphi\mid \varphi\in\ell_2(V,m)\}\;.
\]
Similar formulas hold in transportation networks replacing the discrete gradient by the gradient with Dirichlet conditions on the sink and source vertices.
Finally, we give formulas to compute the capacity of a cut and the value of a flow in terms
of $\cL^-$ and $d$, respectively. In this case the out Laplacian refers to the weighted Laplacian with the capacity function of the network as a weight on the arcs.\\[1mm]

\paragraph{\bf Notation:}
 $G=(V,A,\partial)$ denotes a finite digraph with vertices $V$, arcs $A$ and incidence operator $\partial$. We use the index $+$ (resp. $-$) to denote the \emph{in} (resp. \emph{out}) components of the graph or operator on it. For example, $A_v^+$ (respectively $A_v^-$) denotes the arcs ending (resp. starting) at the vertex $v$. We allow multiple arcs and loops and we say that $G$ is a multidigraph or, simply, a digraph. 
 If $M$ is an operator associated to a finite graph $G$ of order $n$ we denote its spectrum by 
 $\sigma(M):=\{\lambda_1,\dots,\lambda_n\}$ taking multiplicities into account. 
 If $\sigma(M)\in \R$, e.g. when $M$ is selfadjoint (or symmetric matrix), we write the multiset of its eigenvalues in increasing order, i.e., $\lambda_1\leq \dots\leq \lambda_n$, and taking multiplicities into account. 

\paragraph{\bf Acknowledgements:} It is a pleasure to thank Carlos Marijuan for many  conversations on digraphs and an anonymous referee for useful suggestions on the first version of the article.

\section{Weighted digraphs and components of the Laplacian}
\label{sec:graph.theory}

In this section, we introduce, to fix our notation, the main definitions and results 
for weighted multidigraphs (or digraphs for short) and standard operators on them, like the adjacency and the Laplacian operators. These operators can be naturally obtained from the evaluation maps, the discrete gradient, and their adjoints. Standard references are \cite{chung97,bondy-murty:08}.

\subsection{Digraphs and weights}
\label{subsec:graph_theo}
We begin introducing formally the type of graphs we consider in this article.
\begin{definition} 
\label{def:disc_graphs}
A \textit{(discrete) multidigraph} or, simply, a digraph $G$ is given by the tuple $(V, A, \partial)$, where $V$ and $A$ are finite disjoint sets of vertices and arcs (multiple arcs and loops allowed), respectively. The function $\partial\colon A \longrightarrow V \times V$ is the \textit{incidence map}: $\partial a=\left(\partial_{-} a, \partial_{+} a\right)$ is the pair of initial and terminal vertices of $a \in A$. In particular, $\partial_{\pm} a$ fixes an orientation of $a$. For a vertex $v \in V$, we define by
$$
A_{v}^{\pm}:=\left\{a \in A \mid \partial_{\pm} a=v\right\}
$$
the set of arcs ending (also denoted by $(\text{in}/+)$) and starting 
(denoted by $(\text{out}/-)$) at $v$, respectively.\footnote{This notation is not uniform in the literature. In fact, in \cite{bondy-murty:08} the index $^+$ is used to denote outgoing arcs.} 
Moreover, we refer to the set of all arcs adjacent to \textit{v} as $A_{v}=A_{v}^{+} \sqcup A_{v}^{-}$
and write the degree of the vertex $v\in V$ as
$$
\operatorname{deg}_{G} v:=\operatorname{deg} v=|A_{v}|\;.
$$
Similarly, we denote by $\operatorname{deg}^{+/-}v=\left|A_{v}^{+/-}\right|$ the in/out degree, i.e., the set of arcs ending/starting in $v$. For subsets $B$ and $C$ of $V$ we write
$$
A^{+/-}(B, C):=\left\{a \in A \mid \partial_{-/+} a \in B, \partial_{+/-} a \in C\right\}
$$ 
for the set of arcs from $B$ to $C$ $(+)$ and from $C$ to $B$ $(-)$. In particular, $A_{v}^{\pm}=A^{\pm}(V, v)$ and $A^{+}(B, C)=A^{-}(C, B)$, where, for simplicity, $A^{\pm}(V, v)$ stands for $A^{\pm}(V,\{v\})$. If both orientations are taken into account we avoid any index and simply write
$$
A(B, C)=A^{+}(B, C) \cup A^{-}(B, C).
$$
We denote by $N_{v}:=\{u\in V\mid A(u,v) \subset A\}$ the neighborhood of $v\in V$. If we need to specify the orientation, we write $N_{v}^+$ and $N_{v}^-$, for the ingoing and outgoing neighborhoods, respectively.
\end{definition}
    
The set of non-isolated vertices of digraph $G$ may be labeled as $V_{+/-}$ denoting the  \textit{in}/\textit{out} vertices, i.e., vertices with arcs ending in or starting from them, i.e. 
\begin{equation}\label{in-out-vertices}
V_{+/-}:= \{v \in V \mid \exists \, a \in A \text{ with } v=\partial_{+/-}a \}= \{\partial_{+/-}a\mid a \in A\}\;.
\end{equation}

Later we will need to specify the following substructures of a digraph: a \textit{subdigraph} $G^{\prime}=\left(V^{\prime}, A ^{\prime}, \partial^{\prime}\right)$ of $G$ is given by subsets $V^{\prime} \subset V$ and $A^{\prime} \subset A$ with $A^{\prime} \subset$ $A^{G}\left(V^{\prime}, V^{\prime}\right)$ (i.e., $\partial a \in V^{\prime} \times V^{\prime}$ for all $\left.a \in A^{\prime}\right)$ and $\partial^{\prime}=\partial\! \restriction_{A\prime}$. Furthermore, an \textit{induced subdigraph} is a subdigraph with $A^{G}\left(V^{\prime}, V^{\prime}\right)=A^{\prime}$ (i.e., if $v_{1}, v_{2} \in V^{\prime}$, then also all arcs in the original digraph $G$ have to be in 
$G^{\prime}$). Important subdigraphs in this article will be the sources and sinks defined in Definition~\ref{def:sink_source}.

\begin{example} \label{exa:graphs_1}
To fix notation we consider the following representation of a digraph as given in Fig.~\ref{fig:e_2_2}. If two vertices are joined by an arc we use an arrow to indicate its incidence. If two vertices are joined by two arcs with opposite directions we represent it simply with an edge (see the edge joining the vertices $v_2,v_3$). 
	\begin{figure}[h]
		\centering
				\begin{tikzpicture}[baseline,vertex/.style={circle,draw,fill, thick, inner sep=0pt,minimum size=2mm},scale=0.9]
					\node (1) at (1,1) [vertex,label=above:$v_1$] {};
					\node (2) at (2.2,1) [vertex,label=above:$v_2$] {};
					\node (3) at (3.4,1) [vertex,label=above:$v_3$] {};
					\node (4) at (5,2) [vertex,label=above:$v_4$] {};
					\node (5) at (5,0) [vertex,label=above:$v_5$] {};	
					\node[draw=none] (6) at (3,-0.1) {};	
					\node[draw=none] (6) at (3,3) {};
					\draw[-{Latex[ width=1.4mm]}](1) edge node[above] {$a_1$} (2);
					\draw (2) edge node[above] {$_{a_2,a_3}$} (3);
					\draw[-{Latex[ width=1.4mm]}](3) edge node[above] {$a_4$} (4);			
					\draw[-{Latex[ width=1.4mm]}](3) edge node[above] {$a_5$} (5);
					\draw[-{Latex[ width=1.4mm]}](2) to [out=330,in=270,looseness=20] node[below] {$a_6$} (2) ;
				\end{tikzpicture}
		\caption{A representation of a digraph of order $5$.}
		\label{fig:e_2_2}
	\end{figure}
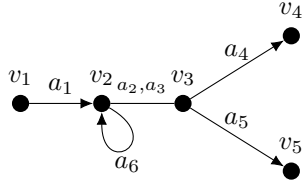
\end{example}

\begin{definition}\label{def:weight_func}  (\textit{Weight functions}) A \textit{weighted (discrete) digraph} $(G, m)$ is a pair given by a digraph $G=(V, A, \partial)$ and two functions (called vertex and arc weights), which are denoted by the same symbol
$$
 m\colon V \longrightarrow(0, \infty) \quad \text { and } \quad m\colon A \longrightarrow(0, \infty)\;.
$$
Note that weights can be thought as positive discrete measures on vertices and arcs. In particular, we denote for subsets $U\subset V$ and $B\subset A$
\[
 m(U):=\sum_{v\in U} m(u)\quadtext{and}  m(B):=\sum_{a\in B} m(a)\;.
\]
The in/out relative weights $\mathrm{rel}_{m}^{+/-} \colon V \longrightarrow(0, \infty)$ 
and the undirected relative weight $\mathrm{rel}_{m} \colon V \longrightarrow(0, \infty)$
are defined respectively as
\begin{equation}\label{eq:rel}
\operatorname{rel}_{m}^\pm(v):=\frac{1}{m(v)} \sum_{a \in A^\pm_{v}} m(a)=\frac{m\left(A^\pm_{v}\right)}{m(v)}, \quadtext{and} 
\operatorname{rel}_{m}(v):= (\operatorname{rel}_{m}^+  + \operatorname{rel}_{m}^-)(v).
\end{equation}
\end{definition}

\begin{example}\label{ex:comb_w} The weight $m=1$ (i.e., $m(v)=1$ and $m(a)=1$ for all $v \in V$ and $a \in A$) is called \textit{combinatorial weight}. For the weighted digraph with combinatorial weight, we also write $G^{\mathrm{comb}}=(G, 1)$. In this case, the relative weight is just the unoriented degree, i.e., $\operatorname{rel}_{1}(v)=\operatorname{deg} v$ for all $v \in V$.
\end{example}

\subsection{First and second order operators on digraphs}
\label{subsec:graph_theo_operators}

We consider in this subsection various types of second order operators which may be constructed from first order evaluation maps. We begin introducing the following natural Hilbert spaces on weighted graphs. These operator definitions will later be translated into matrix representations, yielding each of the well-known adjacency, Laplacian, incidence and degree matrices in graph theory.

\begin{definition}\label{def:hil_spaces} Let $(G, m)$ be a weighted digraph with $G=(V, A, \partial)$. Denote the associated Hilbert spaces
	$$
	\begin{aligned}
		&\ell_{2}(V, m):=\{\varphi\colon V \longrightarrow \mathbb{C} \}\quadtext{and} \ell_{2}(A, m):=\{\eta\colon A \longrightarrow \mathbb{C}\}.
	\end{aligned}
	$$
	with inner products and norms defined as usual in terms of the weights by
	$$
	\langle\varphi, \psi\rangle_{\ell_{2}(V, m)}=\sum_{v \in V} \varphi(v) \overline{\psi(v)} m(v) \text { and }\langle\eta, \alpha\rangle_{\ell_{2}(A, m)}=\sum_{a \in A} \eta(a) \overline{\alpha(a)} m(a) .
	$$
	$$
	\|\varphi\|_{\ell_{2}(V, m)}^{2}:=\sum_{v \in V}|\varphi(v)|^{2} m(v) \quad \text { and } \quad\|\eta\|_{\ell_{2}(A, m)}^{2}:=\sum_{a \in A}|\eta(a)|^{2} m(a) .
	$$
\end{definition}

Next we introduce the canonical evaluation maps which will, in several natural combinations, lead to the corresponding second order operators.
\begin{definition}(Evaluation maps and discrete gradient) \label{def:eva_der}
Let $(G, m)$ be a weighted digraph with $G=(V, A, \partial)$. 
For $\varphi \in \ell_{2}(V, m)$ the \textit{in/out evaluation maps} $d^\pm \varphi\colon A \longrightarrow \mathbb{C}$ 
$$
d^\pm\colon\ell_{2}(V, m)\to \ell_{2}(A, m), \,\, (d^\pm \varphi)(a)=\pm\varphi\left(\partial_{\pm} a\right) ,\quad  a \in A.
$$
Moreover, the evaluation maps correspond to the natural decomposition of the well-known \textit{discrete derivative} $d$ which can be interpreted as a discrete gradient. We refer to 
Subsection~\ref{sec:flows} for a relation between $d$ and circulations of the digraph.
$$
d\colon \ell_{2}(V, m)\to  \ell_{2}(A, m)\quadtext{with}
(d \varphi)(a):= ((d^+ + d^-)\varphi )(a) = \varphi\left(\partial_{+} a\right)-\varphi\left(\partial_{-} a\right) ,\quad  a \in A.
$$
\end{definition}

A straightforward application of the handshaking lemma gives the following expression for the corresponding adjoint operators
$(d^{\pm})^{*}\colon \ell_{2}(A, m) \longrightarrow \ell_{2}(V, m)$ :
\begin{equation}\label{eq:adjoint-d}
	\left((d^{\pm})^{*}\eta\right)(v)=\pm\frac{1}{m(v)} \sum_{a \in A_{v}^{\pm}}m(a) \eta(a)\; .
\end{equation}    

From these formulas one recovers, by linearity, the usual expressions of the adjoint of the discrete derivative $d^{*}\colon \ell_{2}(A, m) \rightarrow \ell_{2}(V, m)$ (see, e.g., \cite[Section~2]{lledo-post:08}) which has a nice interpretation as a (weighted) discrete divergence:
\begin{equation}\label{eq:div}
	\left(d^{*} \eta\right)(v)=\frac{1}{m(v)} \sum_{a \in A_{v}} \vec{\eta}_{a}(v) m(a), \quad \text { where } \quad \vec{\eta}_{a}(v)
	                                 := \begin{cases}
	                                     \eta(a), & v=\partial_{+} a, \\ 
 									-\eta(a), & v=\partial_{-} a\;.
									\end{cases}
\end{equation}

\begin{example}\label{ex:deg_op_comb} 
The simplest second order (discrete) operators that can be obtained from the first order evaluation maps are the in/out degree operators on a weighted graph $(G,m)$. In fact, a straightforward computation gives for $\varphi \in \ell_{2}(V, m)$ the following multiplication operators
$$
		\mathcal{D}^{\pm} := (d^\pm)^{*}d^\pm \colon \ell_{2}(V, m) \longrightarrow \ell_{2}(V, m) 
		\quadtext{with}
		\left(\mathcal{D}^\pm \varphi \right)(v)= \operatorname{rel}_{m}^\pm(v) \varphi(v)\;.
$$
The unoriented degree operator is defined as $\mathcal{D}\equiv \mathcal{D}_{(G,m)} := \mathcal{D}^+ + \mathcal{D}^-$. Note that for combinatorial weights, the expression reduces to the usual multiplication operator with the vertex degree, i.e., $(\mathcal{D}_{\mathrm{comb}} \, \varphi)(v)= \operatorname{deg}(v) \varphi(v)$, $v\in V$ (recall Example~\ref{ex:comb_w}).
\end{example}

\begin{definition}\label{def:adj_op}(Adjacency operators and its directed components) 
The in/out-adjency operator $\mathcal{A}^{+/-}$  on a weighted digraph $(G,m)$ is defined by 
		$$
		\mathcal{A}^{\pm}  := -(d^\pm)^{*}d^\mp\colon \ell_{2}(V, m) \longrightarrow \ell_{2}(V, m) \;.
		$$
		Its action on functions $\varphi \in \ell_{2}(V, m)$ is easily seen to be
		$$
			\left(\mathcal{A}^\pm\varphi\right)(v)=\frac{ 1}{m(v)} \sum_{a \in A_{v}^\pm}\varphi(\partial_{\mp}a) m(a).\\
		$$ 
		Note that the operators $\mathcal{A}^\pm$ are, in general, non-selfadjoint, while the usual unoriented adjacency operator $\mathcal{A}$ on $G$, defined in terms of its in/out components as $\mathcal{A}:= \mathcal{A}^+ + \mathcal{A}^-$, is selfadjoint. 
\end{definition}

Finally, we introduce the directed components of the Laplacian, which mix the discrete derivative with the corresponding evaluation maps.
\begin{definition}\label{def:lap_op} (The Laplacian and its directed components) The in/out-Laplace operators 
$\mathcal{L}^{+/-}$ on a weighted digraph $(G,m)$ are defined by 
$$
\mathcal{L}^{\pm} :=(d^\pm)^{*} d\colon \ell_{2}(V, m) \longrightarrow \ell_{2}(V, m) \;.
$$
The unoriented Laplacian $\mathcal{L}$ on $G$ is defined in terms of its components as 
$\mathcal{L} := \mathcal{L}^+ + \mathcal{L}^-$ which is selfadjoint while its in/out components need not be selfadjoint.
\end{definition}

We make contact with the usual expression of the weighted Laplacian (see, e.g., 
\cite[p.~2]{athansiadis:96} in the case of combinatorial weights).
Since these weighted Laplacians are central in this article, we include a proof of the following statement for completeness. 

\begin{proposition}\label{prop:lap_d_exp}
Let $\varphi \in \ell_2(V,m)$, then
$$
\left(\mathcal{L}^\pm \varphi\right)(v)=\left((\mathcal{D}^\pm-\mathcal{A}^\pm)\varphi\right)(v)
=\frac{1}{m(v)} \sum_{a \in A_{v}^\pm}\left(\varphi(v)-\varphi\left(v_{a}\right)\right) m(a)=\operatorname{rel}^\pm_{m}(v) \varphi(v)-\frac{1}{m(v)} \sum_{a \in A^\pm_{v}} \varphi\left(v_{a}\right) m(a).
$$
Moreover, $0$ is an eigenvalue of $\mathcal{L}^\pm$ with (nonzero) constant eigenfunction.
\end{proposition}
\begin{proof}
First note that by the definition of the involved operators in terms of first order derivatives, we have
\[
 \mathcal{L}^\pm
    = \left(d^\pm\right)^*d
    = \left(d^\pm\right)^*d^+ +\left(d^\pm\right)^*d^-
    =\mathcal{D}^\pm-\mathcal{A}^\pm\;.
\]
Moreover, the explicit expression can be obtained via
	$$
	\begin{aligned}
		(\mathcal{L}^\pm  \varphi)(v)
		&=\frac{\pm1}{m(v)} \sum_{a \in A_{v}^\pm}({d \varphi})(a) m(a) 
		   = \frac{\pm1}{m(v)} \sum_{a \in A_{v}^\pm}( \varphi(\partial_+ a) - \varphi(\partial_- a) ) m(a)\\
		&= \frac{1}{m(v)} \sum_{a \in A_{v}^\pm}( \varphi(v) - \varphi(v_a) ) m(a)
		   =\operatorname{rel}^\pm_{m}(v) \varphi(v)-\frac{1}{m(v)} \sum_{a \in A^\pm_{v}} \varphi\left(v_{a}\right) m(a)\;,
	\end{aligned}
	$$
	where we use the notation $v_a$ to identify the vertex opposed to $v\in V$ along the arc $a\in A$. Furthermore, note that any constant function is in the kernel of the discrete derivative $d$. Therefore, by definition of the in/out Laplacians, $0$ is an eigenvalue of $\mathcal{L}^\pm$ because $\ker d\subset \ker \mathcal{L}^\pm$.
\end{proof}

An immediate consequence of the preceding result is the relation to the usual weighted Laplacian. In fact, the in/out Laplacians can be understood as natural components of $\mathcal{L}$.
\begin{corollary}\label{cor:lap_exp}Let $(G, m)$ be a weighted digraph $G=(V, A, \partial)$ and $\mathcal{L} = \mathcal{L}^+ + \mathcal{L}^-$ the weighted Laplacian. For $\varphi \in \ell_2(V,m)$ we have
	$$
    (\mathcal{L} \varphi)(v)=((\mathcal{L}^+  + \mathcal{L}^- )\varphi)(v)
		=\frac{1}{m(v)} \sum_{a \in A_{v}}\left(\varphi(v)-\varphi\left(v_{a}\right)\right) m(a)
		=\operatorname{rel}_{m}(v) \varphi(v)-\frac{1}{m(v)} \sum_{a \in A_{v}} \varphi\left(v_{a}\right) m(a)\;.
	$$
\end{corollary}
Note that if $G$ is unoriented and the weights of the corresponding edge/arcs weights satisfy $m(a)=m(\overline{a})=m(e)/2$ then $\mathcal{L}^+=\mathcal{L}^-=\mathcal{L}/2$.

\begin{remark}\label{rem:ONB}
For concrete calculations, it is useful to represent operators as matrices. For this, given a weighted digraph $(G,m)$ with 
$G=(V, A, \partial)$ fix the following canonical orthonormal basis (ONB) 
$\left\{\delta_{v} \mid v \in V\right\}$ of $\ell_{2}(V, m)$ 
(respectively $\left\{\delta_{a} \mid a \in A\right\}$ of $\ell_{2}(A, m)$) where
$$
\delta_{v}(w):= 
\begin{cases}
1 / m(v)^{1 / 2}, & v=w, \\ 0, & v \neq w
\end{cases}
\quad \text{and} \quad 
\delta_{a}(a'):= 
\begin{cases}
1 / m(a)^{1 / 2}, & a'=a, \\ 0, & a' \neq a\;.
\end{cases}  
$$
Note that $\left\langle\delta_{v}, \delta_{w}\right\rangle_{\ell_{2}(V, m)}=m(v) /\left(m(v)^{1 / 2} m(v)^{1 / 2}\right)=1$ 
for $v=w$ and $0$ otherwise and similarly for $\{\delta_{a}\mid a\in A\}$.

For this basis, the adjacency and Laplacian matrices are given by
\begin{enumerate}
    \item \textit{In, out and unoriented incidence matrices:}
	$$
\left(\mathcal{B}^\pm\right)_{aw}:=\langle  \delta_{v}, d^\pm\delta_{w}\rangle 
		=  \begin{cases}\pm\sqrt{\frac{m(a)}{m(w)}}, & \text{ if } \partial_{\pm} a= w, \\
			0, & \text{other}.\end{cases} \text{, }
		\left(\mathcal{B}\right)_{aw}:=\left\langle \delta_{v}, d  \delta_{w}\right\rangle 
		= \begin{cases}
            +\sqrt{\frac{m(a)}{m(w)}}, & \text{ if } \partial_{+} a= w, \\
            -\sqrt{\frac{m(a)}{m(w)}}, & \text{ if } \partial_{-} a= w, \\
			0, & \text{other}.
        \end{cases}
	$$
	\item \textit{In, out and unoriented adjacency matrices:}
	$$
\left(\mathcal{A}^\pm\right)_{v w}:=\langle  \delta_{v}, \mathcal{A}^\pm\delta_{w}\rangle 
		= \begin{cases}\frac{ m(A^\pm(v, v))}{m(v)}, & v=w, \\
			\frac{m(A^\pm(w, v))}{\sqrt{m(v) m(w)}}, & v \neq w.
          \end{cases} \text{, }
		\left(\mathcal{A}\right)_{v w}:=\left\langle \delta_{v}, \mathcal{A}  \delta_{w}\right\rangle 
		= \begin{cases}\frac{2 m(A(v, v))}{m(v)},  &  v=w, \\
			\frac{m(A(w, v))}{\sqrt{m(v) m(w)}}, & v \neq w,
          \end{cases}
	$$
	\textit{where we set} \[ 
	              m\left(A^{\prime}\right):=\sum_{a \in A^{\prime}} m(a)\quadtext{(loops counted once; see Definition~\ref{def:weight_func})}.
	             \]

	\item \textit{In, out and unoriented Laplacian matrices:}
	$$
		\left(\mathcal{L}^\pm\right)_{v w}:=\langle \delta_{v}, \mathcal{L}^\pm \delta_{w}\rangle 
		= \begin{cases}
		  \frac{m\left(A_{v}^\pm\right)-m(A^\pm(v, v))}{m(v)} & v=w, \\
		  -\frac{m(A^\pm(w, v))}{\sqrt{m(v) m(w)}}, & v \neq w.
		  \end{cases} \text{, }
		\left(\mathcal{L}\right)_{v w}:=\left\langle \delta_{v}, \mathcal{L} \delta_{w}\right\rangle 
		   =\begin{cases}\frac{m\left(A_{v}\right)-2 m(A(v, v))}{m(v)}, & v=w, \\
			-\frac{m(A(w, v))}{\sqrt{m(v) m(w)}}, & v \neq w \;.
          \end{cases}
	$$
\end{enumerate}
\end{remark}

\begin{example}\label{exa:w_digrph}
We illustrate the different components of the adjacency and Laplacian matrices in the example of the digraph $(G,m)$ in Fig.~\ref{fig:exa:w_digrph}, with weights denoted by numbers over arcs and vertices that also number them. The corresponding matrices are
	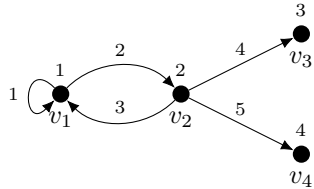
\begin{figure}[h]
	    \begin{tikzpicture}[baseline,vertex/.style={circle,draw,fill, thick, inner sep=0pt,minimum size=2mm},scale=0.8]
					\node (1) at (1,0) [vertex,label=above:$_1$, label=below:$v_1$] {};
					\node (2) at (3,0) [vertex,label=above:$_2$, label=below:$v_2$] {};
					\node (3) at (5,1) [vertex,label=above:$_3$, label=below:$v_3$] {};
					\node (4) at (5,-1) [vertex,label=above:$_4$, label=below:$v_4$] {};
					\draw[-{Latex[ width=1.4mm]}](1) to [out=45,in=135,looseness=1] node[above] {$_2$} (2) ;
					\draw[-{Latex[ width=1.4mm]}](2) to [out=225,in=315,looseness=1] node[above] {$_{3}$} (1) ;
					\draw[-{Latex[ width=1.4mm]}](2) edge node[above] {$_4$} (3);
					\draw[-{Latex[ width=1.4mm]}](2) edge node[above] {$_5$} (4);
					\draw[-{Latex[ width=1.4mm]}](1) to [out=135,in=225,looseness=10] node[left] {$_1$} (1) ;
		\end{tikzpicture}
		\caption{Weighted digraph.} 
		\label{fig:exa:w_digrph}
	\end{figure}
	
	\begin{tabular}{l l l l}
		    {In adjacency matrix} &{$\mathcal{A^+} = 
				\begin{bmatrix*}
					1 & \frac{3}{\sqrt{2}} & 0 & 0 \\[1mm]
					\sqrt{2} & 0 & 0 & 0\\[1mm]
					0 & \frac{4}{\sqrt{6}} & 0 & 0 \\[1mm]
					0 & \frac{5}{2\sqrt{2}} & 0 & 0 \\[1mm]
				\end{bmatrix*} $,}
        & Out adjacency  & \!\!\!\!\!matrix $\mathcal{A}^- = (\mathcal{A}^+)^*  $. \\[35pt]
{In Laplacian} &
		{$\mathcal{L^+} = 
			\begin{bmatrix*}[r]
					2 & \shortminus\frac{3}{\sqrt{2}} & 0 & 0 \\[1mm]
					\shortminus\sqrt{2} & 1 & 0 & 0 \\[1mm]
					0 & \shortminus\frac{4}{\sqrt{6}} & \frac{4}{3} & 0  \\[1mm]
					0 & \shortminus\frac{5}{2\sqrt{2}} & 0 & \frac{5}{4} 
				\end{bmatrix*}$,}
& {Out Laplacian} &
{$\mathcal{L^-} = 
\begin{bmatrix*}[r]
 1 & \shortminus \sqrt{2} & 0 & 0 \\[1mm]
 \shortminus\frac{\sqrt{3}}{\sqrt{2}} & \frac{3}{2} 
& \shortminus \frac{\sqrt{4}}{\sqrt{6}} & \shortminus \frac{\sqrt{5}}{\sqrt{8}} \\[1mm]
 0 & 0 & 0 & 0  \\[1mm]
 0 & 0 & 0 & 0 
\end{bmatrix*}$.}
				\\[35pt]
{Adjacency matrix} &
		{$\mathcal{A}\,\,\, =  
				\begin{bmatrix*}[r]
					2 & \frac{5}{\sqrt{2}} & 0 & 0 \\[1mm]
					\frac{5}{\sqrt{2}} & 0 & \frac{4}{\sqrt{6}} & \frac{5}{2\sqrt{2}} \\[1mm]
					0 & \frac{4}{\sqrt{6}} & 0 & 0  \\[1mm]
					0 & \frac{5}{2\sqrt{2}} & 0 & 0 
				\end{bmatrix*}  $,}& {Laplacian}&
		{$\mathcal{L}\,\, = 
			\begin{bmatrix*}[r]
					5 & \shortminus\frac{5}{\sqrt{2}} & 0 & 0 \\[1mm]
					\shortminus\frac{5}{\sqrt{2} }& 7 & \shortminus\frac{4}{\sqrt{6}} & \shortminus\frac{5}{2\sqrt{2}} \\[1mm]
					0 & \shortminus\frac{4}{\sqrt{6}} & \frac{4}{3} & 0 \\[1mm]
					0 & \shortminus\frac{5}{2\sqrt{2}} & 0 & \frac{5}{4}  
				\end{bmatrix*}   $.}\\
		\end{tabular}
	
\end{example}

\begin{remark}
In \cite{mulas:24} the authors define non-backtracking (di)graphs and the corresponding (non-symmetric) Laplacians
$\cL_\mathrm{NB}$. 
These matrices show nice spectral properties (see also \cite{mulas:23}). For example, Table~2 in \cite{mulas:23}
suggests that these operators have a stronger spectral rigidity in relation to other notions of Laplacian.
Exploring a reformulation of non-backtracking Laplacians in the geometrical framework introduced in this article could yield insightful results. In particular, one could study if the method of construction of infinite families of isospectral graphs labeled by partitions of natural numbers also applies in the non-backtracking framework (see \cite{fclp:23,fclp:22}).
\end{remark}

\section{Sources, sinks and the spectrum of the in/out Laplacians} 
\label{sec:connectedness}

As a motivation of the main result in this section consider the out Laplacian 
$\mathcal{L^-}$ of the digraph represented in Fig.~\ref{fig:exa:w_digrph} (Example~\ref{exa:w_digrph}). Its eigenvalues (taking multiplicities into account) 
are given by 
\[ 
   \sigma\left( \mathcal{L^-}\right)=\{-0.5,0,0,3\}
\]
and note that the multiplicity of the zero eigenvalue coincides with the number of sinks (in this example, vertices with zero out-degree) labeled by $v_3$ and $v_4$. We will establish in full generality that the number of sources/sinks of a digraph with arbitrary weights is characterized by the multiplicity of $0$ in the spectrum of  $\mathcal{L^{+/-}}$. Similar results establishing the connection between the multiplicity of $0$ and related notions like, for example, the \emph{reaches} of a graph or the minimal number of directed trees needed to span the graph can be found in 
\cite{veerman:06,bauer:12}. 

We begin introducing standard notions of connectivity in digraphs 
(see \cite{b-brualdi:92,bondy-murty:08,bang-jensen-gutin:18}). We illustrate these concepts in Examples~\ref{exa:chains} and \ref{exa:sss}.
\begin{definition}\label{def:dir_path}(Directed paths) 
A \textit{directed path} $P$ of length \textit{p} in a digraph $G$ is a sequence $v_0,a_1,v_1,a_2,...,v_{p-1},a_p$, $v_p$ of vertices $\{v_j\}_{j=0}^{p}$ and arcs $\{a_j\}_{j=1}^{p}$ verifying 
$(\partial_+a_j = v_j \text{ and } \partial_-a_j = v_{j-1})$ such that $v_i\neq v_j$ if $i\neq j$. In this case, it is said that $v_0$ is \textit{directedly connected} with $v_p$, in symbols, $v_0 \rightsquigarrow v_p$ (or $v_p$ is reachable from $v_0$). The digraph $G$ is said to be \textit{directedly connected} or, simply, \textit{$d$-connected} if for any pair $u,v \in V$ there is at least one directed path either $u \rightsquigarrow v$ or $v \rightsquigarrow u$. The graph $G$ is said to be \textit{strongly connected} if for any pair $u,v \in V$ there exist both directed paths $u \rightsquigarrow v$ and $v \rightsquigarrow u$.
Finally, a \textit{strongly connected component} (SCC) of $G$ is a maximal induced subdigraph (cf., Definition~\ref{def:disc_graphs}) which is strongly connected. 
\end{definition}

It follows immediately from the previous definition that a digraph is strongly connected if and only if it only has one SCC. Also, note that any vertex is strongly connected to itself and the SCCs of a digraph give a partition of the vertex set of $G$. Finally, it is also clear that if $G$ is strongly connected then $G$ is also $d$-connected. It is also immediate that the $d$-connected relation $u \rightsquigarrow v$ specifies a preorder on the set of vertices of the graph, i.e., a reflexive and transitive relation. If $u \rightsquigarrow v$ we say $u$ is smaller than $v$. The preorder naturally sugests the notion of {\em chain} and 
{\em maximal chain} within the vertex set that can also be used to characterize $d$-connectedness of a digraph.

\begin{definition}\label{def:prechains}(Chain and maximal chain) Given the preorder relation $\rightsquigarrow $ in 
$ V\times V$, a \textit{chain} is a subset of vertices $V ' \subset V$ verifying that for any $ u,v\in V' $ either
$u \rightsquigarrow v$ or $v \rightsquigarrow u$ (i.e., any pair of vertices in $V'$ is comparable).
Furthermore, a chain $V ' \subset V$ is called a \textit{maximal chain} if there is no other chain $V''\subset V$ such that $V'\varsubsetneq V''$. A \textit{lower/upper bound} of a maximal chain $V' \subset V$ is a vertex $v\in V'$ smaller/greater than any other vertex in $V'$.
\end{definition}

Note that the reachable set $\cR(v)$ of a vertex $v$ as mentioned in \cite{veerman:19} may be an example of a chain in a digraph. Nevertheless, these two notions are different in general. Take, for example, a star graph emitting from the center vertex into the leaves. The (maximal) chains are the pair of vertices consisting of the center vertex and the leaves. On the other hand the reachable set of the center vertex is the whole vertex set of the star graph.

\begin{proposition}\label{prop:d_con_one_mp}
    Let $(G, m)$ be a weighted digraph $G=(V, A, \partial)$. Then $G$ is $d$-connected if and only if there is a unique maximal chain in V.
\end{proposition}
\begin{proof}
    The implication ($\Longrightarrow$) follows from the Definition~\ref{def:dir_path}. To show the reverse implication
    let $V'$ be the maximal chain in $V$. Choose $u,v \in V$ and assume that there is no directed path joining $u$ with $v$ and $v$ with $u$. Then at least one vertex (say $u$) is not in $V'$. Take now a maximal chain $V''$ containing $u$ which, by construction, is different from $V'$. This contradicts the fact that there is only one maximal chain. Therefore, $V' =V $ and $G$ must be \textit{$d$-connected}.
\end{proof}

\begin{example}\label{exa:chains} Consider the digraph $(G,m)$ given in Fig.~\ref{fig:exa:chains} which has four maximal chains. 
This example shows that the maximal chains do not specify a partition of the vertex set. The lower (respectively upper) bound of these chains determine subsets $F$ (respectively $S$) of vertices where arcs do not point in (respectively point out) from the complement $F^c$ (respectively into the complement $S^c$). Concretely we have two upper and lower bounds given by $F_1=\{v_1\}$, $F_2=\{v_2\}$ and $S_1=\{v_5\}$, $S_2=\{v_6\}$.

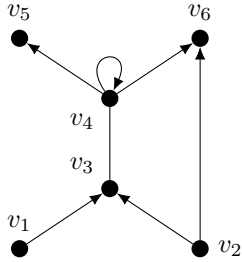
\begin{figure}[h]
    \centering
    \begin{tabular}{m{6cm} m{6cm}}
        {\begin{tikzpicture}[baseline,vertex/.style={circle,draw,fill, thick, inner sep=0pt,minimum size=2mm},scale=0.8]
    						\node (1) at (1.5,0) [vertex,label=above:$v_1$] {};
    						\node (2) at (4.5,0) [vertex,label=right:$v_2$] {};
    						\node (3) at (3,1) [vertex,label=135:$v_3$] {};
    						\node (4) at (3,2.5) [vertex,label=225:$v_4$] {};
    						\node (5) at (1.5,3.5) [vertex,label=above:$v_5$] {};
    						\node (6) at (4.5,3.5) [vertex,label=above:$v_6$] {};
    						\draw[-{Latex[ width=1.4mm]}](1) edge node {} (3);
    						\draw[-{Latex[ width=1.4mm]}](2) edge node{} (3);
    						\draw(3) edge (4);
    						\draw[-{Latex[ width=1.4mm]}](4) edge node{} (5);
    						\draw[-{Latex[ width=1.4mm]}](4) edge node{} (6);
    						\draw[-{Latex[ width=1.4mm]}](2) edge node{} (6);
    						\draw[-{Latex[ width=1.4mm]}](4) to [out=125,in=65,looseness=15] node[] {} (4) ;
    	\end{tikzpicture}}&
    	{\vspace{6mm} The four maximal chains are given by
    	\begin{enumerate}
    	    \item $v_1\rightsquigarrow \{v_3,v_4\} \rightsquigarrow v_5$,
    	    \item $v_1\rightsquigarrow \{v_3,v_4\} \rightsquigarrow v_6$,
    	    \item $v_2\rightsquigarrow \{v_3,v_4\} \rightsquigarrow v_5$,
    	    \item $v_2\rightsquigarrow \{v_3,v_4\} \rightsquigarrow v_6$.
    	\end{enumerate}}
    \end{tabular}
    \caption{A digraph with four maximal chains. 
    Lower/upper bounds given by $\{v_1\}$,$\{v_2\}$ and $\{v_5\}$,$\{v_6\}$.}
    \label{fig:exa:chains}
\end{figure}
\end{example}

Motivated by the previous example we introduce the following formal definition source and sink of a digraph which will be related to 
the zero eigenvalue of $\mathcal{L}^-$ and $\mathcal{L}^+$, respectively (see also \cite{bjorner-lovasz:92,veerman:06} for related ideas).

\begin{definition}\label{def:sink_source}(Sink, source and stream) Let $(G, m)$ be a weighted digraph $G=(V, A, \partial)$. A \textit{source} $F$ of $G$ is the set of vertices $F\subset V$ of a SCC of $G$ such that $A^-(F,F^c) = \varnothing$ (no arcs entering $F$). Similarly, a \textit{sink} $S$ of $G$ is the set of vertices $S\subset V$ of a SCC of $G$ such that $A^+(S,S^c) = \varnothing$ (no arcs leaving $S$). Finally, the \textit{stream} of $G$ is the set of vertices $T\subset V$ of $G$ that do not belong to a sink nor a source of $G$.
Sometimes, sinks, sources and streams are referred to a given chain $V'$. In that case, we will use the notation $S_{V'}$, $F_{V'}$ and $T_{V'}$, respectively.
\end{definition}
Note that sources, sinks and streams refer to the set of vertices and not to the corresponing induced subgraphs. Moreover, a connected unoriented graph (i.e., each edge comes with both orientations) has a unique SCC and the vertex set is both a sink and a source. In this case, the stream is empty.

We consider next some straightforward facts related to the notion of maximal chains:

\begin{proposition}\label{prop:ss_mp}
Let $(G, m)$ be a weighted digraph $G=(V, A, \partial)$ and $V'\subset V$ a maximal chain.
\begin{enumerate}
\item The subgraph $G'$ induced by $V'$ (see Definition~\ref{def:disc_graphs}) has a unique sink $S_{V'}$ and a unique source $F_{V'}$.
\item Any vertex of $v' \in S_{V'}$ is an upper bound of $V'$. Conversely, any vertex of $v' \in F_{V'}$ is a lower bound of $V'$.
\end{enumerate}
\end{proposition}
\begin{proof}
To show (i) assume that $S_1'\subset V'$ and $S_2'\subset V'$ are different sinks of $G'$. Since both are SCCs, we have $S_1' \cap S_2' =\varnothing $. 
By definition of a sink, there is no directed path connecting $s_1\in S_1'$ and $s_2\in S_2'$ or vice versa, contradicting the fact that $V'$ is a chain. For similar reasons, $F_{V'}$ is unique as well. Part (ii) is immediate since any $v \in S_{V'}$ satisfies $u \rightsquigarrow v$
for any $u \in V'$. Similarly, any vertex in $F_{V'}$ is a lower bound of $V'$.
\end{proof}

\begin{example}\label{exa:sss}
We illustrate the notions introduced in this section in the following examples.
\begin{enumerate}
\item The oriented path $P_n$, $n\geq 2$, is $d$-connected but not strongly connected. In fact, it has $n$ SCCs given by each vertex. It has a unique maximal chain $V = V(P_n)$ with source $F_V = \{v_1\}$, sink $S_V = \{v_n\}$ and stream $T = \{v_2,v_3,...,v_{n-2},v_{n-1}\}$.
            \begin{figure}[h]
        		\begin{tikzpicture}[baseline,vertex/.style={circle,draw,fill, thick, inner sep=0pt,minimum size=2mm},scale=0.8]
						\node (1) at (-0.2,0) [vertex,label=above:$v_1$] {};
						\node (2) at (1.3,0) [vertex,label=above:$v_2$] {};
						\node (3) at (2.8,0) [vertex,label=above:$v_3$] {};
						\node (4) at (4.7,0) [vertex,label=above:$v_{n-2}$] {};
						\node (5) at (6.2,0) [vertex,label=above:$v_{n-1}$] {};
						\node (6) at (7.7,0) [vertex,label=above:$v_n$] {};
						\draw[-{Latex[ width=1.4mm]}](1) to node[] {} (2) ;
						\draw[-{Latex[ width=1.4mm]}](2) to node[] {} (3) ;
                        \draw (3)--($(3)!1/4!(4)$) ;
                        \draw[dotted] ($(3)!1/4!(4)$) -- ($(3)!2/3!(4)$);
						\draw ($(3)!2/3!(4)$)[-{Latex[ width=1.4mm]}] to node[] {} (4);
						\draw[-{Latex[ width=1.4mm]}](4) to node[] {} (5) ;
						\draw[-{Latex[ width=1.4mm]}](5) to node[] {} (6) ;
    			\end{tikzpicture}
        		\caption{Representation of the oriented path of length n.}
        		\label{fig:exa:sss:p_n}
        	\end{figure}
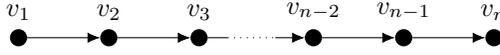
\item The oriented cycle $C_n$, $n\geq3$, is strongly connected and $C_n$ is its unique SCC. Moreover, 
        it has a unique maximal chain $V = V(C_n)$ where source and sink coincide, i.e. $F_V = S_V = V$, and has an empty stream, i.e. $T =\varnothing$.
            \begin{figure}[h]
        		\begin{tikzpicture}[baseline,vertex/.style={circle,draw,fill, thick, inner sep=0pt,minimum size=2mm},scale=0.6]
					\node (2) at ( 60:2cm) [vertex,label=above:$v_2$] {};
                    \node (3) at (120:2cm) [vertex,label=above:$v_3$] {};
                    \node (4) at (180:2cm) [vertex,label=100:$v_{n-2}$] {};
                    \node (5) at (240:2cm) [vertex,label=left:$v_{n-1}$] {};
                    \node (6) at (300:2cm) [vertex,label=360:$v_{n}$] {};
                    \node (1) at (360:2cm) [vertex,label=80:$v_1$] {};
					\draw[-{Latex[ width=1.4mm]}](1) to node[] {} (2) ;
					\draw[-{Latex[ width=1.4mm]}](2) to node[] {} (3) ;
                    \draw (3)--($(3)!1/4!(4)$) ;
                    \draw[dotted] ($(3)!1/4!(4)$) -- ($(3)!2/3!(4)$);
					\draw ($(3)!2/3!(4)$)[-{Latex[ width=1.4mm]}] to node[] {} (4);
					\draw[-{Latex[ width=1.4mm]}](4) to node[] {} (5) ;
					\draw[-{Latex[ width=1.4mm]}](5) to node[] {} (6) ;
					\draw[-{Latex[ width=1.4mm]}](6) to node[] {} (1) ;
    			\end{tikzpicture}
        		\caption{Representation of the cyclic digraph $C_n$.}
        		\label{fig:exa:sss:c_n}
        	\end{figure}
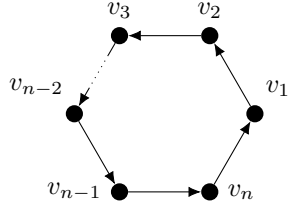   
\item The following digraph $G$ of order $8$ given in Fig.~\ref{fig:exa:sss} is neither strongly connected nor $d$-connected (note, for example, that there is no directed path joining $w_1$ with $w_2$). It has four SCCs given by 
$C_1=\{u_1,u_2,u_3\}$, $C_2=\{v_1,v_2,v_3\}$, $C_3=\{w_1\}$ and $C_4=\{w_2\}$. Moreover, $G$ has two maximal chains
given by 
\[
 V_1' = \{u_1,u_2,u_3\}\rightsquigarrow \{v_1,v_2,v_3\}\rightsquigarrow w_1\quadtext{and}
 V_2' = \{u_1,u_2,u_3\}\rightsquigarrow \{v_1,v_2,v_3\}\rightsquigarrow w_2
\]
with corresponding sources, sinks and stream given respectively by 
$F_{V_1'} = F_{V_2'} = C_1$, $S_{V_1'} = C_3$, $S_{V_2'} = C_4$, $T = \{v_1, v_2, v_3\}$.
            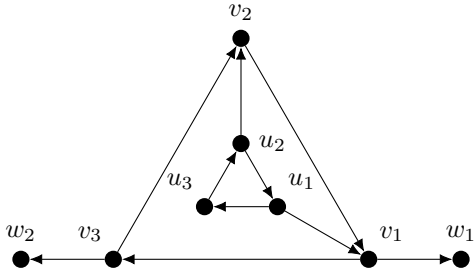
\begin{figure}[h]
                \centering
                    {\begin{tikzpicture}[baseline,vertex/.style={circle,draw,fill, thick, inner sep=0pt,minimum size=2mm},scale=0.7]
         					    \node (1) at (-30:0.8cm) [vertex,label={80:$u_1$}] {};
                                \node (2) at (90:0.8cm) [vertex,label=right:$u_2$] {};
                                \node (3) at (210:0.8cm) [vertex,label=100:$u_3$] {};
                                \node (4) at (-30:2.8cm) [vertex,label=80:$v_1$] {};
                                \node (5) at (90:2.8cm) [vertex,label=above:$v_2$] {};
                                \node (6) at (210:2.8cm) [vertex,label=100:$v_3$] {};
                                \node [right=of 4,vertex,label=above:$w_{1}$] (7) {};
                                \node [left=of 6,vertex,label=above:$w_{2}$] (8) {};
                                \draw[-{Latex[ width=1.4mm]}](1) to node[] {} (3) ;
                                \draw[-{Latex[ width=1.4mm]}](3) to node[] {} (2) ;
                                \draw[-{Latex[ width=1.4mm]}](2) to node[] {} (1) ;
                                \draw[-{Latex[ width=1.4mm]}](2) to node[] {} (5) ;
                                \draw[-{Latex[ width=1.4mm]}](1) to node[] {} (4) ;
                                \draw[-{Latex[ width=1.4mm]}](4) to node[] {} (6) ;
                                \draw[-{Latex[ width=1.4mm]}](6) to node[] {} (5) ;
                                \draw[-{Latex[ width=1.4mm]}](5) to node[] {} (4) ;
                                \draw[-{Latex[ width=1.4mm]}](6) to node[] {} (8) ;
                                \draw[-{Latex[ width=1.4mm]}](4) to node[] {} (7) ;
                    \end{tikzpicture}}
                \caption{Representation of digraph $G$ with $2$ maximal chains.}
                \label{fig:exa:sss}
            \end{figure}
\end{enumerate}
\end{example}

\begin{proposition}\label{prop:ss_existence} Let $(G, m)$ be a finite weighted digraph, then $G$ has at least a pair $F$ and $S$ of source and sink connected by a directed path.
\end{proposition}
\begin{proof}
The case where the digraph has no arcs or has a unique SCC (like an oriented cycle) which is both sink and source holds trivially. Consider now all directed paths that leave a source $F$ which exists since the graph has at least one arc. The lower bounds of these paths, which exist since the graph is finite, specify the corresponding sinks which, by construction, are connected to the vertices of $F$. 
\end{proof}

The next definition introduces a topological concept that is associated to one sink and one source of a finite digraph $G$.
\begin{definition}\label{def:dir_comp}(Directed component) 
Consider a finite weighted digraph $(G,m)$ and a pair $F$ (source) and $S$ (sink) connected by a directed path. Let $V'\subset V$ be the set of vertices of the all directed paths from $F$ to $S$. We say that a \textit{directed component} is the digraph $G'$ induced by $V'$.
\end{definition}

\begin{remark}\label{rem:dir_com}
Given a set of sources $\{F_i'\}_{i=1}^f$ and sinks $\{S_i'\}_{i=1}^s$ of a digraph $(G,m)$ the following results can be easily established:
\begin{enumerate}
\item $G$ can be decomposed in (not necessarily disjoint) directed components 
$\left\{\vec{G_j'}\right\}_{j=1}^{c}$, i.e. $G = \bigcup_j \vec{G_j'}$, with $f \leqslant c$ and $s \leqslant c$.
\item \label{rer_com:decom} If $F_j \neq S_i$ for $j=1,2,...,f$ and $i=1,2,...,s$. Then $V$ can be decomposed as 
\[ 
  V = \left(\cup_i S_{V_i'}\right) \sqcup \left(\cup_i F_{V_i'}\right) \sqcup T\;,
\]
where $\sqcup$ means disjoint union.
\item \label{rem:dir_com:dcomp_fs} $G$ has one $d$-component if and only if $G$ only contains one sink and one source.
\end{enumerate}
\end{remark}

\begin{proposition}\label{prop:d-connected}
If the graph $G$ is $d$-connected, then $G$ only has one sink and one source, i.e., $G$ has one directed component.
\end{proposition}
\begin{proof}
	If $G$ has more than one sink or source, then there would be more than one maximal chain. But this leads to a contradiction by Proposition~\ref{prop:d_con_one_mp}.
\end{proof}
\begin{example}\label{exa:counter_exa_1s1f}
	Note that the reverse implication in the preceding proposition is not always true. As a counterexample consider the following digraph given in Figure~\ref{fig:exa:counter_exa_1s1f}, with one sink $\{v_4\}$, one source $\{v_1\}$ and one 
	$d$-component. Note that $G$ is not $d$-connected as the vertices $\{v_2\}$ and $\{v_3\}$ are not connected by a directed path.
	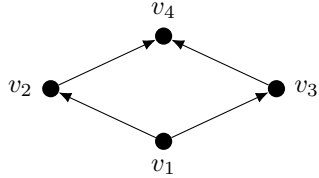
\begin{figure}[h]\label{fig:deledge}
		\centering
		\captionsetup{width=0.5\linewidth}
		{ 
			\begin{tikzpicture}[baseline,vertex/.style={circle,draw,fill, thick, inner sep=0pt,minimum size=2mm},scale=1]
				\node (1) at (1,0) [vertex,label=below:$v_1$] {};
				\node (2) at (-0.5,0.7) [vertex,label=left:$v_2$] {};
				\node (3) at (2.5,0.7) [vertex,label=right:$v_3$] {};
				\node (4) at (1,1.4) [vertex,label=above:$v_4$] {};		
				\draw[-{Latex[ width=1.4mm]}](1) edge node[below] {} (3);
				\draw[-{Latex[ width=1.4mm]}](1) edge node[below] {} (2);
				\draw[-{Latex[ width=1.4mm]}](3) edge node[below] {} (4);
				\draw[-{Latex[ width=1.4mm]}](2) edge node[left] {} (4);
		\end{tikzpicture} } 
		\caption{Example of a digraph with one $F=\{v_1\}$ and sink $S=\{v_4\}$ that is not $d$-connected.}
		\label{fig:exa:counter_exa_1s1f}
	\end{figure}
\end{example}

In the following definition we will introduce the notion of compression of an operator to a subset of vertices of the graphs. This notion will be applied to describe different properties of the in/out Laplacians. In spectral graph theory compressions are very useful when describing Laplacians with Dirichlet conditions on a selected subset of vertices (see, for example, \cite{bauer:12}). The spectrum of Dirichlet Laplacians naturally appears when developing spectral localization (\emph{bracketing}) techniques of Laplacians on periodic graphs or in the recent construction method of isospectral graphs (see \cite{lledo-post:08,fclp:23}). See also \cite{balti:17} for the definition of a non-symmetric Dirichlet operator on digraphs.

\begin{definition}(Compression of an operator)\label{def:compression}
Let $\mathcal{K}\colon \ell_{2}(V,m) \to\ell_{2}(V,m) $ be a linear operator and $V' \subset V$ be a subset of vertices. Then the \textit{compression} of $\mathcal{K}$ to $V'$ is the operator 
$\mathcal{K}_{V'}\colon \ell_{2}(V',m) \to\ell_{2}(V',m) $, given by
$$
\mathcal{K}_{V'} = \iota_{V'}^* \circ \mathcal{K} \circ \iota_{V'}\;,
$$
where $\iota_{V'}\colon \ell_{2}(V', m) \to \ell_{2}(V, m)$ the natural subspace embedding and
the corresponding adjoint $\iota_{V'}^*\colon \ell_{2}(V, m) \to \ell_{2}(V', m)$ is the orthogonal projection onto the subspace $\ell_{2}(V', m)$. Relabeling, if necessary, the vertex set one obtains the block structure
$$
\mathcal{K} \cong 
\left( \begin{array}{c|c}
  \mathcal{K}_{V'} & * \\
  \hline
                 * & *
\end{array}\right).
$$
\end{definition}

Note that, in general, the compression of a Laplacian to a subset of vertices is not the Laplacian of a subgraph. But this property is true if we compress the components $\mathcal{L}^+$ or $\mathcal{L}^-$ respectively to sources ($F$) or sinks ($S$) of the digraph (see also Example~\ref{exa:spec_par}) since sources/sinks have no arc entering~$F$/leaving~$S$ from their respective complements.

\begin{proposition}\label{prop:compression_on_SF} Let $(G,m)$ be a weighted digraph. 
Then, the compression of the in-component of the Laplacian $\mathcal{L}^+$ of $G$ to a source $F$ is the Laplacian operator $\mathcal{L}^+_F$ for the induced digraph $G_F$ generated by $F$. The same result holds for the out-Laplacian $\mathcal{L}^-$ compressed to a sink $S$ of $G$.
\end{proposition}
\begin{proof}
We focus only on the case of $\mathcal{L}^+$ since the case of the out-Laplacian $\mathcal{L}^-$ is done similarly. By Remark~\ref{rem:ONB} each entry $(i,j)$ of in components of $\mathcal{L}^+$ is the weighted measure of the arcs leaving $v_j$ and entering $v_i$ if $i\neq j$, or the weighted measure of the arcs entering $v_i$ (loops removed) if $i=j$. Since sources of $G$ have no entering arcs from vertices in its complement (see Definition~\ref{def:sink_source}), the only arcs incident to a vertex in $F$ must be in $F$ as well. In other words, the compression $\mathcal{L}_F^+$ represents the arcs connecting the vertices $v\in F$, and the diagonal represents these same connections since the arcs leaving the source vertices into its complement are not considered in the rows where the block $\mathcal{L}_F^+$ is placed. That is, 
$\mathcal{L}^+_F$ represents the Laplacian operator for the induced digraph $G_F$ generated by $F$.
\end{proof}

\begin{remark}
A more general result applies to the in/out adjacency operators $\mathcal{A^\pm}$. In fact, in this case 
the compressed matrices to \emph{any} subset of vertices correspond to the adjacency matrix of the corresponding induced subgraph.
To see this, consider the matrix representation of the in/out adjacency matrices and observe that now the diagonal only contains information about the loops at each vertex. Thus, the matrix compression to any subset of vertices can be performed without the restrictions needed in the case of the components of the corresponding Laplacians (see also Example~\ref{exa:w_digrph}).
\end{remark}

The preceding results lead to the following block triangular structure of $\cL^\pm$ and the corresponding spectral decomposition.

\begin{proposition}\label{prop:spec_par}(Decomposition of the spectrum) 
Let $(G,m)$ be a weighted digraph with $G=(V, A, \partial)$ and denote by $s\in \N$ (respectively $f\in \N$) the number of sinks (respectively sources) of $G$. Consider
$F = \sqcup_{i=1}^f F_i$, where $F_i \subset V$ is a source of $G$, $S = \sqcup_{i=1}^s S_i$, where $S_i \subset V$ is a sink of $G$ and, finally, $T$ the stream of $G$. Then,
\begin{enumerate}
\item  $\sigma(\mathcal{L}_G^+) = \bigcup\limits_{i=1}^{f} \sigma(\mathcal{L}_{F_i}^+) \cup \sigma(\mathcal{L}_{T \sqcup S}^+)$.
\item  $\sigma(\mathcal{L}_G^-) = \bigcup\limits_{i=1}^{s} \sigma(\mathcal{L}_{S_i}^-) \cup \sigma(\mathcal{L}_{T \sqcup F}^-)$.
\end{enumerate}
\end{proposition}
\begin{proof}
It is enough to show (i) since the case (ii) is done similarly. Furthermore, we will only consider the nontrivial scenario in which $G$ possesses at least one non-coincident source and sink.
Then like in Remark~\ref{rem:dir_com}~, one can partition the vertex set as 
$ V = (T\sqcup S) \sqcup (\sqcup_i F_i)$, which gives the following upper triangular matrix:
$$
\mathcal{L}_G^+ = 
\left( \begin{array}{c|c|c|c|c}
\mathcal{L}_{T \sqcup S}^+ &* &  \cdots & \cdots & *\\
\hline
0& \mathcal{L}_{F_1}^+ &0 & \cdots & 0 \\
\hline
0& 0&  \mathcal{L}_{F_2}^+  & \ddots & \vdots\\
\hline
\vdots& \vdots& \ddots & \ddots & 0 \\
\hline
0& 0& \hdots & 0 & \mathcal{L}_{F_f}^+\\
\end{array}\right) 
\vspace{3mm}
$$
Note that the row blocks of a source $F_i$ of $G$ (excluding $\mathcal{L}_{F_i}^+$) are zero, since by Definition~\ref{def:sink_source} there are no arcs connecting the set of vertices $F_i^c$ to $F_i$. Then, we have
$$
\det(\mathcal{L}_G^+-\lambda\1) = 
                \det\left(\mathcal{L}_{T \sqcup S}^+-\lambda\1\right)
                \cdot \prod\limits_{i=1}^{f} 
                \det\left(\mathcal{L}_{F_i}^+-\lambda\1\right) \;,
$$
which implies the statement on the spectrum.
\end{proof}

\begin{example}\label{exa:spec_par}
We give some examples of the block decomposition for the out-Laplace matrix 
$\mathcal{L}^-$ and relate its spectrum to the spectrum of the given compressions.
\begin{enumerate}
    \item Consider the digraph $G_1$ given by Fig.~\ref{fig:exa:spec_par_1_1} with combinatorial weights (cf. Example~\ref{ex:comb_w}) and the corresponding out-Laplacian $\mathcal{L}^-_{G_1}$
        \begin{figure}[h]
            \begin{tabular}{p{5cm} c}
                	{\begin{tikzpicture}[baseline,vertex/.style={circle,draw,fill, thick, inner sep=0pt,minimum size=2mm},scale=1.2]
            		    \node (1) at (0,-1.1) [vertex,label=above:$v_1$] {};
                        \node (2) at (-1.1,0) [vertex,label=left:$v_2$] {};
                        \node [above=of 2,vertex,label=left:$v_{3}$](3) {};
                        \node (4) at (1.1,0) [vertex,label=left:$v_4$] {};
                        \node [above=of 4,vertex,label=left:$v_{5}$](5) {};
        				\draw [-{Latex[ width=1.4mm]}](1) to node[above] {} (2) ;
        				\draw (2) edge node[above] {} (3);
        				\draw [-{Latex[ width=1.4mm]}](1) to node[above] {} (4) ;
        				\draw (4) edge node[above] {} (5);
		            \end{tikzpicture}}&{
		            $\mathcal{L}^-_{G_1} = 
			            \begin{bmatrix*}[r]
        					2 & \shortminus1 & 0 & \shortminus1 & 0\\
        					0 & 1 & \shortminus1 & 0 &0\\
        					0 & \shortminus1 & 1 & 0  &0\\
        					0 & 0 & 0 & 1 & \shortminus1\\
        					0 & 0 & 0 & \shortminus1 & 1
				        \end{bmatrix*} . $
		            }
	            \end{tabular}
		    \caption{Digraph representation of $G_1$ and $\mathcal{L}^-_{G_1}$.}
		    \label{fig:exa:spec_par_1_1}
    	\end{figure}
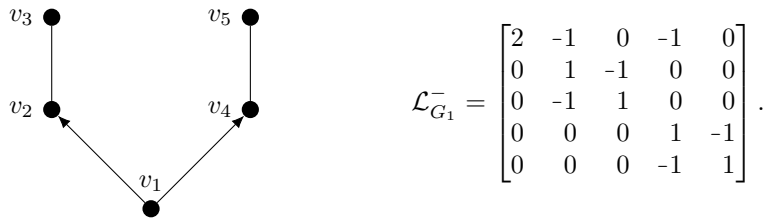
    	
Then, we have $\sigma(\mathcal{L}_{G_1}^-) = \{0,0,2,2,2\}$. Considering the compression to $F = \{v_1\}$, $S_1 = \{v_2,v_3\}$ and $S_2 = \{v_4,v_5\}$ we obtain
\[
\mathcal{L}_F^- = 
	            \begin{bmatrix*}[r]
                2
                \end{bmatrix*} 
\quadtext{and}
\mathcal{L}_{S_1}^- ={L}_{S_2}^-=
	     \begin{bmatrix*}[r]
             1 & \shortminus1\\
              \shortminus1 & 1 \\
        \end{bmatrix*} \;,
\]
verifying $\sigma(\mathcal{L}_{F}^-) = \{2\}$ and $\sigma(\mathcal{L}_{S_1}^-) = \sigma(\mathcal{L}_{S_2}^-)= \{0,2\}$.
    	
    \item Consider the digraph $G_2$ given in Fig.~\ref{fig:exa:spec_par_1_2} with combinatorial weights
    and with out-Laplace matrix $\mathcal{L}^-_{G_2}$:
        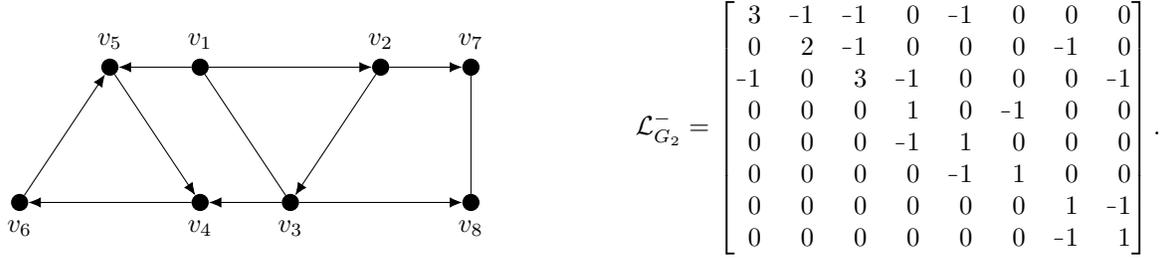
\begin{figure}[h]
            \begin{tabular}{p{8cm} c}
                	{\begin{tikzpicture}[baseline,vertex/.style={circle,draw,fill, thick, inner sep=0pt,minimum size=2mm},scale=0.6]
            		    \node (1) at (-4,1.5) [vertex,label=above:$v_5$] {};
                        \node (2) at (-6,-1.5) [vertex,label=below:$v_6$] {};
                        \node (3) at (-2,-1.5) [vertex,label=below:$v_4$] {};
                        \node (4) at (-2,1.5) [vertex,label=above:$v_1$] {};
                        \node (5) at (2,1.5) [vertex,label=above:$v_2$] {};
                        \node (6) at (0,-1.5) [vertex,label=below:$v_3$] {};
                        \node (7) at (4,1.5) [vertex,label=above:$v_7$] {};
                        \node (8) at (4,-1.5) [vertex,label=below:$v_8$] {};
                        \draw [-{Latex[ width=1.4mm]}](1) to node[above] {} (3) ;
                        \draw [-{Latex[ width=1.4mm]}](3) to node[above] {} (2) ;
                        \draw [-{Latex[ width=1.4mm]}](2) to node[above] {} (1) ;
                        \draw [-{Latex[ width=1.4mm]}](4) to node[above] {} (1) ;
                        \draw [-{Latex[ width=1.4mm]}](6) to node[above] {} (3) ;
                        \draw [-{Latex[ width=1.4mm]}](4) to node[above] {} (5) ;
                        \draw [-{Latex[ width=1.4mm]}](5) to node[above] {} (6) ;
        				\draw (4) edge node[above] {} (6);
        				\draw [-{Latex[ width=1.4mm]}](5) to node[above] {} (7) ;
                        \draw [-{Latex[ width=1.4mm]}](6) to node[above] {} (8) ;
                        \draw (7) edge node[above] {} (8);
		            \end{tikzpicture}}&{
		            $\mathcal{L}^-_{G_2} = 
			            \begin{bmatrix*}[r]
        					3 & \shortminus1 & \shortminus1 & 0 & \shortminus1 & 0 & 0 & 0\\
        					0 & 2 & \shortminus1 & 0 & 0 & 0 & \shortminus1 & 0\\
        					\shortminus1 & 0 & 3 & \shortminus1 & 0 & 0 & 0 & \shortminus1\\
        					0 & 0 & 0 & 1 & 0 & \shortminus1 & 0 & 0 \\
        					0 & 0 & 0 & \shortminus1 & 1 & 0 & 0 & 0 \\
        					0 & 0 & 0 & 0 & \shortminus1 & 1 & 0 & 0 \\
        					0 & 0 & 0 & 0 & 0 & 0 & 1 & \shortminus1\\
        					0 & 0 & 0 & 0 & 0 & 0 & \shortminus1 & 1\\
				        \end{bmatrix*} . $
		            }
	            \end{tabular}
		    \caption{Digraph representation of $G_2$ with out-Laplacian $\mathcal{L}^-_{G_2}$.}
		    \label{fig:exa:spec_par_1_2}
    	\end{figure}
     
The spectrum of the out-Laplace matrix $\mathcal{L}^-_{G_2}$:
$$
\sigma(\mathcal{L}^-_{G_2}) =
      \left\{0,2,3, \frac{5\pm\sqrt{5}}{2}, \frac{3\pm i\sqrt{3}}{2}\right\}\;.
$$

Taking now into account the compression to the blocks for $F = \{v_1,v_2,v_3\}$, $S_1 = \{v_4,v_5,v_6\}$ and $S_2 = \{v_7,v_8\}$ and proceeding in a similar manner as in the previous example, we obtain 
$$\sigma(\mathcal{L}_{F}^-)   = \left\{\frac{5\pm\sqrt{5}}{2},3\right\},\; 
  \sigma(\mathcal{L}_{S_1}^-) = \left\{0,\frac{3\pm i\sqrt{3}}{2}\right\} 
\quadtext{and} \sigma\left(\mathcal{L}_{S_2}^-\right)= \{0,2\}\;.$$
\end{enumerate}
\end{example}

Note that in preceding examples, out-Laplace matrices compressed to sinks share the eigenvalue 0. We conclude this section by proving a theorem that establishes an important relationship between the number of sources/sinks of a digraph and the multiplicity of the zero eigenvalue of $\mathcal{L}^{+/-}$. To simplify the proof, we first state the following properties of eigenfunctions $\varphi \in \operatorname{ker}(\mathcal{L}^{\pm})$.

The next result gives several useful characterizations of the functions in the kernel of the in/out Laplacians. For example, its values at a vertex are given by the \textit{oriented weighted mean} over the in/out neighbors of $v$ (see also \cite{veerman:06}).

\begin{lemma}\label{lem:phi_in_ker}
    Let $(G, m)$ be a weighted digraph $G=(V, A, \partial)$, $\mathcal{L}^{\pm} = (d^\pm)^{*} d$ the corresponding in$/$out-Laplace operator and $\varphi \in \ell_{2}(V, m)$ a function defined on the set of vertices $V$.
    \begin{enumerate}
        \item $\varphi \in \operatorname{ker}(\mathcal{L}^{\pm}) \iff \varphi(v) = \frac{1}{m(A^\pm_{v})} \sum_{a \in A^\pm_{v}} \varphi(v_{a}) m(a)$, $v\in V^{+/-}$, where  $m\left(A^{\pm}_v\right):=\sum_{a \in A^{\pm}_v} m(a)$ and $v_a$ denotes the vertex opposed to $v\in V$ along the arc $a\in A$.
        
        \item \label{lem:phi_in_ker:max}If $G$ has a unique source$/$sink denoted by $F/S$ and $\varphi \in \operatorname{ker}(\mathcal{L}^{\pm})$, then $\varphi(v)$ is constant on $F/S$ and the maximum of $\{|\varphi(v)| \mid v \in V\}$ is attained on $F/S$.
    \end{enumerate}
\end{lemma}
\begin{proof}
To show (i) let $\varphi \in \operatorname{ker}(\mathcal{L}^\pm)$. Then, for any $\tau \in \ell_{2}(V,m)$
$$
\begin{aligned}
0 = \langle \mathcal{L}^\pm\varphi, \tau \rangle_{\ell_{2}(V,m)} 
        &= \sum_{v\in V} \left(\frac{1}{m(v)}\sum_{a \in A^\pm_v}(\varphi(v) - \varphi(v_a)) m(a)\right)\overline{\tau(v)} m(v) \\
        &=\sum_{v\in V} \sum_{a \in A^\pm_v}(\varphi(v) - \varphi(v_a)) m(a)\overline{\tau(v)}\;.
\end{aligned}
$$
Taking for $\tau$ the indicator functions of the in/out vertices, i.e., 
$\{\1_{w} \mid w \in V^\pm\}$, we obtain $|V^\pm|$ equations verifying
$$
    0          = \sum_{a \in A^\pm_{w}}(\varphi(w) - \varphi(w_{a})) m(a)
    \quadtext{hence}
    \varphi(w) = \frac{1}{m(A^\pm_{w})} \sum_{a \in A^-_{w}} \varphi(w_{a}) m(a)\;,
$$
where $m(A^\pm_{w}) := \sum_{a \in A^\pm_{w}}m(a)$. In other words, the value of $\varphi(v)$, $v \in V^\pm$ is given by the \textit{oriented mean} over the in/out neighbours of $v$.

To prove (ii) let $v_M \in V$ such that $|\varphi(v_M)| = \max_{v\in V}|\varphi(v)|$. If $v_M \in F$ (or $v_M\in S$), then there is nothing to show, so assume $v_M \in V\setminus F$ (or $v_M\in V\setminus S$). Then, by part (i) we have:
    
$$
\begin{aligned}
	|\varphi(v_M)| \leq& \frac{1}{m(A^\pm_{v_M})} \sum_{a \in A^\pm_{v_M}} |\varphi(v_{M,a})| m(a)\\
                   \leq&  \frac{1}{m(A^\pm_{v_M})} \sum_{a \in A^\pm_{v_M}}| \varphi(v_{M})| m(a) = |\varphi(v_M)|.
\end{aligned}
$$
Therefore, 
$$|\varphi(v_M)| =\frac{1}{m(A^\pm_{v_M})} \sum_{a \in A^\pm_{v_M}} |\varphi(v_{M,a})| m(a)\;.$$ 

Since $|\varphi(v_M)|$ is maximal and coincides with the weighted mean of positve numbers we conclude that $|\varphi(v_M)|=|\varphi(v_{M,a})|$ for any $a\in A^\pm_{v_M}$. Moreover,
the triangle inequality used in the first inequality above turns to be an equality only when the summands $\varphi(v_{M,a})$ are pairwise multiples by a non-negative real number. Therefore,
$\varphi(v_M) = \varphi(v)$ for all $v \in N_{v_M}^\pm$ (in/out-neighborhood of $v_M$, Definition~\ref{def:disc_graphs}).
Reasoning recursively from the in/out-neighbors of $v_M$ and since there is only one source/sink in $G$, there will be a vertex $v_1\in F/S$ satisfying $\varphi(v_M) = \varphi(v_1)$. So, $\max_{v\in V}|\varphi(v)|$ is attained on $F/S$. As sources and sinks are by definition SCCs, we conclude similarly that $\varphi(v)$ must be constant on $F/S$.
\end{proof}

We prove next spectral characterization of sources and sinks of general weighted digraphs
(see also \cite{bjorner-lovasz:92,veerman:06,bauer:12}).

\begin{theorem}\label{theo:zero_io_lap}
Let $(G, m)$ be a weighted digraph $G=(V, A, \partial)$ and $\mathcal{L}^{\pm} = (d^\pm)^{*} d$ the corresponding 
in$/$out Laplacians.
    
\centerline{$G$ has one source/sink $\Longleftrightarrow$ $0$ is a simple eigenvalue of $\mathcal{L}_G^{+/-}$.}\vspace{2mm}
    
Moreover, the multiplicity of the eigenvalue 0 of $\mathcal{L}_G^{+/-}$ gives the number of sources$/$sinks of $G$.
\end{theorem}

\begin{proof}
To show the implication ($\Longrightarrow$) we consider only the case of sources since a similar reasoning can be applied to sinks. Recall from Proposition~\ref{prop:lap_d_exp} that the constant function $\1(v)$ is always an eigenfunction of the eigenvalue 0. Now let  $\varphi \in \operatorname{ker}(\mathcal{L}^+)$  and recall from Lemma~\ref{lem:phi_in_ker}~(\ref{lem:phi_in_ker:max}) that the restriction 
of $\varphi$ to $F$ (denoted by $\varphi\!\left. \right|_{F}$)
is constant and that $|\varphi(u)| = \max_{v\in V}\{|\varphi(v)|\}$ for any $u \in F$. Since 
$\varphi \!\left. \right|_{F} \neq 0$, we can define a new eigenfunction $\xi(v)$ as the following linear combination of the eigenfunctions $\varphi$ and $\1$: fixing a $u_0\in F$ consider
$$
\xi(v) := \1(v) - \frac{\1(u_0)}{\varphi(u_0)} \varphi(v). 
$$
By construction, $\xi(u_0)=0$ and from Lemma~\ref{lem:phi_in_ker}~(\ref{lem:phi_in_ker:max}) we conclude 
$\xi(v)=0$, $v\in V$. This shows that $\varphi(v) =\varphi(u_0)$, $v\in V$, is constant as well and $0$ is a simple eigenvalue of $\mathcal{L}_G^+$.

To prove the reverse implication ($\Longleftarrow$) suppose $G$ has sources denoted as $F_1,F_2,...,F_c$ with $c \geq 2$. Then by Proposition~\ref{prop:spec_par},
\[ 
\sigma(\mathcal{L}_G^+) = \bigcup\limits_{i=1}^{c} \sigma(\mathcal{L}_{F_i}^+) \cup \sigma(\mathcal{L}_{R}^+),\quadtext{where} 
R = V\setminus \left(\bigcup\limits_{i=1}^{c}F_i\right) \;.
\]

From Proposition~\ref{prop:compression_on_SF}, each $\mathcal{L}_{F_i}^+$ represent the in-Laplace operator for each digraph induced by $F_i$. Thus, every $\sigma(\mathcal{L}_{F_i}^+)$ contains $0$ with constant eigenfunction supported on $F_i$. We conclude that eigenvalue 0 has at least multiplicity $c \geqslant 2$. This proves the first part of the theorem. For the last sentence, we only need to show that eigenvalue 0 has multiplicity $c$. It is enough to show that $0\not\in \sigma(\mathcal{L}_{R}^+)$. Let $\mathcal{L}_G^+$ be represented by

$$
\mathcal{L}^+ = 
\left( \begin{array}{c|c|c|c|c}
	\mathcal{L}_{R}^+ &* &  \cdots & \cdots & *\\
	\hline
	0& \mathcal{L}_{F_1}^+ &0 & \cdots & 0 \\
	\hline
	0& 0&  \mathcal{L}_{F_2}^+  & \ddots & \vdots\\
	\hline
	\vdots& \vdots& \ddots & \ddots & 0 \\
	\hline
	0& 0& \hdots & 0 & \mathcal{L}_{F_c}^+\\
\end{array}\right).
\vspace{3mm}
$$

\noindent Assume $0 \in \sigma(\mathcal{L}_{R}^+)$, then there is an eigenfunction $\psi \in \ell_{2}(V, m)$ that can be extended to an eigenfunction $\hat{\psi} \in \ell_{2}(V, m)$. A matrix representation  would be the following\vspace{-2mm}
$$
\mathcal{L}^+ \hat{\psi} = \left( \begin{array}{c|c}
	\mathcal{L}_{R}^+ &* \\
	\hline
	0& *\\
\end{array}\right)\cdot\left( \begin{array}{c}
	\psi\\
	\hline
	0 \\
\end{array}\right) = 0.
$$
But using again that the maximum of $|\hat{\psi}|$ is attained on a source of $G$ (cf. Lemmas~\ref{lem:phi_in_ker}~(\ref{lem:phi_in_ker:max})) we conclude that $\hat{\psi} = 0$, contradicting the fact that $\psi$ is an eigenfunction.
\end{proof}

The following result is direct consequence of 
Remark~\ref{rem:dir_com}~(\ref{rem:dir_com:dcomp_fs}) and Theorem~\ref{theo:zero_io_lap}.
\begin{corollary}
G has one directed component if and only if $0$ is a simple eigenvalue of both $\mathcal{L}_G^+$ and $\mathcal{L}_G^-$.
\end{corollary}

We conclude this section illustrating some of the results of the previous results in some concrete examples.

\begin{example}\label{exa:ef_l_d_est} Consider the following digraph $G$ with combinatorial weights 
(i.e., $m=1$) represented in Fig.~\ref{exa:ef_l_d_est:fig} with vertices $\{v_1,\dots,v_7\}$ together with the corresponding matrix representations of the in/out-Laplacians $(\mathcal{L}^{+/-})$ (cf. Remark~\ref{rem:ONB}). 

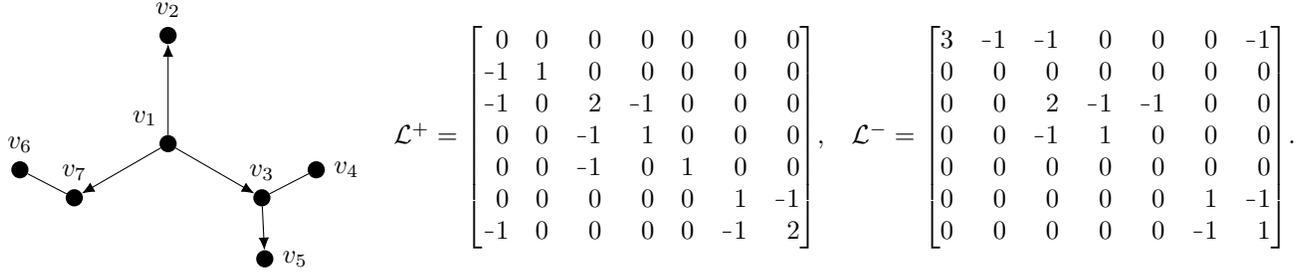
\begin{figure}[h]
		\begin{tabular}{ccc}
		    {\begin{tikzpicture}[baseline,vertex/.style={circle,draw,fill, thick, inner sep=0pt,minimum size=2mm},scale=0.8]
		            \node (1) at ( 60:0cm) [vertex,label=100:$v_1$] {};
					\node (2) at ( 90:1.8cm) [vertex,label=above:$v_2$] {};
                    \node (3) at (-30:1.8cm) [vertex,label=above:$v_3$] {};
                    \node (4) at (-10:2.5cm) [vertex,label=right:$v_4$] {};
                    \node (5) at (-50:2.5cm) [vertex,label=right:$v_5$] {};
                    \node (6) at (190:2.5cm) [vertex,label=above:$v_6$] {};
                    \node (7) at (210:1.8cm) [vertex,label=above:$v_7$] {};
					\draw[-{Latex[ width=1.4mm]}](1) to node[] {} (2) ;
					\draw[-{Latex[ width=1.4mm]}](1) to node[] {} (3) ;
					\draw[-{Latex[ width=1.4mm]}](1) to node[] {} (7) ;
					\draw[-{Latex[ width=1.4mm]}](3) to node[] {} (5) ;
					\draw (3) edge node[above] {} (4);
					\draw (6) edge node[above] {} (7);
    			\end{tikzpicture}}& {$\mathcal{L^+} = 
				\begin{bmatrix*}[r]
					0 & 0 & 0 & 0 & 0 & 0 & 0 \\
					\shortminus1 & 1 & 0 & 0 & 0 & 0 & 0\\
					\shortminus1 & 0 & 2 & \shortminus1 & 0 & 0 & 0 \\
					0 & 0 & \shortminus1 & 1 & 0 & 0 & 0 \\
					0 & 0 & \shortminus1 & 0 & 1 & 0 & 0 \\
					0 & 0 & 0 & 0 & 0 & 1 & \shortminus1 \\
					\shortminus1 & 0 & 0 & 0 & 0 & \shortminus1 & 2  \\
				\end{bmatrix*}  $,} 
            &{$\mathcal{L^-}= 
				\begin{bmatrix*}[r]
					3 & \shortminus1 & \shortminus1 & 0 & 0 & 0 & \shortminus1 \\
					0 & 0 & 0 & 0 & 0 & 0 & 0 \\
					0 & 0 & 2 & \shortminus1 & \shortminus1 & 0 & 0 \\
					0 & 0 & \shortminus1 & 1 & 0 & 0 & 0 \\
					0 & 0 & 0 & 0 & 0 & 0 & 0 \\
					0 & 0 & 0 & 0 & 0 & 1 & \shortminus1 \\
					0 & 0 & 0 & 0 & 0 & \shortminus1 & 1 \\
				\end{bmatrix*} $.}
		\end{tabular}
        \caption{Matrices $(\mathcal{L}^{+}),(\mathcal{L}^{-})$ for a combinatorial digraph.} 
		\label{exa:ef_l_d_est:fig}
\end{figure}
The eigenvalues (taking multiplicities into account) of the in/out Laplacians are given by 
$$
    \sigma((\mathcal{L^+})) \approx \{-1, 0, 0.38, 0.38, 1, 2.62, 2.62\}
    \quadtext{and}
    \sigma((\mathcal{L^-})) \approx \{0, 0, 0, 0.38, 2, 2.62, 3\}\;.
$$
Note that the multiplicity of $0$ in the corresponding spectra coincides with the number of sources and sinks, respectively. In fact, for the graph given in Fig.~\ref{exa:ef_l_d_est:fig} we have one source
$F=\{v_1\}$ and three sinks: $S_1:=\{v_2\}$, $S_2:=\{v_5\}$ and $S_3:=\{v_6,v_7\}$.
\end{example}

As mentioned in the introduction to this article, an important difference in spectral graph theory for digraphs is that the corresponding Laplacians are not selfadjoint and, therefore, their spectra need not be real. We conclude this section showing that acyclic digraphs do have real spectrum. Recall that an {\em acyclic digraph} is a digraph with no directed cycles.

\begin{proposition}(Spectrum of acyclic digraphs) 
Let $(G, m)$ with $G=(V, A, \partial)$ be an acyclic digraph. Then the spectrum of $\mathcal{L}^\pm$ is contained in $[0,\infty)$. Moreover, for combinatorial weights we have 
$\sigma(\mathcal{L}^\pm)\subset \N_0$.
\end{proposition}
\begin{proof}
Given a finite weighted digraph $(G, m)$ we divide the proof into two steps.

First, we show that there is a vertex labeling $V = \{v_i\}_{i=1}^n$ satisfying the following monotony property: if $v_i \rightsquigarrow v_j$ then $i<j$. Note that it is enough to prove this labeling exists for a $d$-component of $G$, since otherwise, one may label vertices consecutively in each $d$-component. 
To show this, note first that if $d$-connected graph $G$ has a cycle then the mentioned labeling is not possible since any pair of vertices of the cycle $v_i$,$v_j$ with $i\not= j$ satisfy $v_i \rightsquigarrow v_j$ \emph{and} $v_j \rightsquigarrow v_i$. Consider a labeling of vertices $\{v_1,\dots,v_n\}$ and define the set 
\[
 \cP:=\{(i,j)\in [n]\times [n] \mid i\not=j\quadtext{and} v_i\rightsquigarrow v_j \}\;.
\]
Note that since $G$ is $d$-connected and acyclic for any pair of distinct indices $(i,j)$ either 
$(i,j)\in\cP$ or $(j,i)\in\cP$ and that if $(i,j)\in\cP$ then, necessarily, $(j,i)\not\in\cP$.
Assume now that for the chosen labeling there is a pair of indices violating monotony and choose 
$(i,j)\in\cP$ such that $i-j>0$ is maximal. 
Consider then a relabeling given by a transposition $\tau$ permuting $i$ with $j$. Note that as $G$ is acyclic  $v_{\tau(i)} \not\rightsquigarrow v_{\tau(j)}$. We analyze now how this maximal difference behaves after relabeling for which we consider the following four cases:

\setlength{\columnsep}{-2cm}
\vspace{-0.2cm}
\begin{multicols}{2}
    \begin{itemize}
        \item For any $v_f \rightsquigarrow v_{\tau(j)}$, $f-\tau(j)=f-i<f-j$.
        \item For any $v_{f'} \rightsquigarrow v_{\tau(i)}$, $f'-\tau(i)=f'-j<i-j$.
        \item For any $v_{t}\in V$ with $v_{\tau(j)} \rightsquigarrow v_{t}$, $\tau(j)-t=i-t<i-j$.
        \item For any $v_{t'}\in V$ with $v_{\tau(i)} \rightsquigarrow v_{t'}$, $\tau(i)-t'=j-t'<i-j$.
    \end{itemize}
\end{multicols}
\vspace{-0.2cm}
Since the set of vertices is finite and after iteration of the process, 
the maximal difference will be lowered consecutively. Therefore, eventually one obtains a labeling satisfying the monotony condition.

Second, recall the definition of the in/out relative weight $\mathrm{rel}^\pm(v)$ given in Eq.~(\ref{eq:rel}) for any vertex $v$. Following Remark~\ref{rem:ONB} and using the previous vertex labeling we obtain that the matrix representations of $\mathcal{L^+}$ and $\mathcal{L^-}$ are lower and upper triangular, respectively:

\begin{figure}[h]
		\begin{tabular}{cc}
		    {$\mathcal{L^+} = 
				\begin{bmatrix*}
					\operatorname{rel^+}(v_1) & 0 & 0 & \hdots & 0  \\
					* & \operatorname{rel^+}(v_2) & 0 &\hdots & 0 & \\
					* & * & \operatorname{rel^+}(v_3) & \hdots & 0  \\
					\vdots &\vdots & \vdots & \ddots & 0  \\
					* & * & * & * & \operatorname{rel^+}(v_n)  \\
				\end{bmatrix*}  $,} 
            &{$\mathcal{L^-} = 
				\begin{bmatrix*}
					\operatorname{rel}^-(v_1) & *  & *  & \hdots & *   \\
					0& \operatorname{rel}^-(v_2) & *  &\hdots & *  & \\
					0& 0& \operatorname{rel}^-(v_3) & \hdots & *   \\
					\vdots &\vdots & \vdots & \ddots & *   \\
					0& 0& 0& 0& \operatorname{rel}^-(v_n)  \\
				\end{bmatrix*}  $.} 
		\end{tabular}
        \caption{Matrices $\mathcal{L}^\pm$ for an acyclic digraph} 
		\label{exa:L_acyclic:fig}
\end{figure}
Since $\operatorname{rel^{\pm}}(v) \in [0,\infty) $
for every $v\in V$ (cf. Eq.~(\ref{eq:rel})) and $\mathcal{L}^\pm$ are lower/upper triangular matrices, the corresponding spectra consist, precisely, of the diagonal entries 
$\operatorname{rel^{\pm}}(v)$ and we conclude that $\sigma (\mathcal{L}^\pm) \subset [0,\infty)$
(see \cite[p.~39]{horn-johnson:90}).
In particular, for combinatorial weights, the spectrum is given by non-negative integers, i.e., 
$\sigma (\mathcal{L}^\pm) \subset \N_0$, since the relative weight is in this case a non-negative integer (see Example~\ref{ex:comb_w}).
\end{proof}

\section{A geometrical interpretation of circulations and flows in networks}\label{sec:flows}

We continue exploring here geometrical consequences from the picture developed in the previous sections. We will generalize to weighted multidigraphs notions like circulations, flows, and related quantities like values of a flow and capacities of a cut.
We will see that these concepts are closely related to discrete divergences and in/out Laplacians $\cL^{+/-}$. For additional results and motivation on these topics we refer to standard references like, e.g., \cite[Chapter~7]{bondy-murty:08} or \cite[Chapter~6]{diestel:17}.

In our language a \emph{circulation} in a digraph $G=(V, A, \partial)$ with combinatorial weights
is a function $\eta\in\ell_2(A)$ satisfying the following conservation law for any $v\in V$
\begin{equation}\label{eq:kirchhoff}
 \sum_{a\in A_{v}^+}\eta(a)=\sum_{a\in A_{v}^-}\eta(a)\;.
\end{equation}
Geometrically, this means that a circulation $\eta$ can be interpreted as a divergenceless function i.e., $d^*\eta=0$ or, equivalently, $d^+\eta(v)=-d^-\eta(v)$, $v\in V$ with the conventions in Subsection~\ref{subsec:graph_theo_operators}. 
This notion has an immediate generalization to weighted digraphs: A \emph{circulation} in a weighted digraph $(G,m)$ with $G=(V, A, \partial)$ is a function $\eta\in\ell_2(A,m)$ satisfying the following conservation law for any $v\in V$
\[
 \sum_{a\in A_{v}^+}m(a)\eta(a)=\sum_{a\in A_{v}^-}m(a)\eta(a)\;,
\]
i.e., a divergenceless function on the set of arcs. If we denote by $\cC$ the set of circulations of a weighted digraph we obtain directly from the definition of discrete divergence in Eq.~(\ref{eq:div}) that $\cC$ is the set of {\em coclosed forms on arcs}, i.e.,
\[
 \cC=\mathrm{ker}d^*\;.
\]
Again, the set of circulations is perpendicular to the set of gradients of functions on the vertex set ({\em exact forms on arcs}), i.e,
\[
 \cC\perp \{ d\varphi\mid \varphi\in \ell_2(V,m) \}\quadtext{since for any} \varphi\in \ell_2(V,m) \quadtext{we have}
  \langle\eta,d\varphi\rangle = \langle d^*\eta,\varphi\rangle=0\;.
\]

A transportation network $N(V_0,V_1)$ is a digraph with two distinguished sets of vertices
$V_0\subset V$ (respectively, $V_1\subset V$) which are sources (respectively, sinks) of the weighted digraph. An $(V_0,V_1)$-flow is a function $\eta\in\ell_2(A,m)$ satisfying the following conservation law on the complement of sources and sinks, i.e. for any 
$v\in V\setminus W$ with $W:=V_0\cup V_1$ we have
\[
 \sum_{a\in A_{v}^+}m(a)\eta(a)=\sum_{a\in A_{v}^-}m(a)\eta(a)\;.
\]
Moreover, the set of $(V_0,V_1)$-flows, which we denote by $\cC_{V_0,V_1}$, has a natural geometrical description. Let
\[
 \iota\colon \ell_2\Big((V\setminus W),m\Big)\hookrightarrow \ell_2(V,m)
\]
be the natural embedding of Hilbert spaces extending by $0$ the functions on $W$. In other words
the space $\ell_2((V\setminus W),m)$ can be identified with the subspace of 
$\ell_2(V,m)$ having Dirichlet conditions on the vertices $W$ (i.e., $\varphi(v)=0$, $v\in W$).
Define the discrete gradient on $\ell_2((V\setminus W),m)$ by the natural composition
\[
 d_0:=d\circ \iota\colon\ell_2\Big((V\setminus W),m\Big)\to\ell_2(A,m)
 \quadtext{with adjoint} 
 d_0^*=\iota^*\circ d^*\colon \ell_2(A,m)\to\ell_2\Big((V\setminus W),m\Big)\;,
\]
where $\iota^*$ is the orthogonal projection onto the subspace $\ell_2((V\setminus W),m)$. Then, a similar computation as before, shows that a $(V_0,V_1)$-flow $\eta$ is just an element of 
$\mathrm{ker}d_0^*$ and
\[
 \cC_{V_0,V_1}\perp \{ d_0\varphi\mid \varphi\in \ell_2((V\setminus W),m) \}\;.
\]

Transportation networks are usually defined together with a so-called \textit{capacity} $c$ 
which is a non-negative real-valued function defined on the set of arcs $A$. This function may be 
interpreted as the maximal rate at which some quantity can flow through the arc. In other words, the function defines the upper limit of the flow rate through the arc. In this context, a flow $\eta$ is \textit{feasible} if it takes non-negative values and does not exceed the maximal amount that can be transported through a given arc, i.e., $0\leq \eta(a) \leq c(a), \, a \in A$. Mathematically, we will interpret the capacity as a weight on the arcs (see Definition~\ref{def:weight_func}) and relate important quantities like values or capacities of flows in networks in terms of the corresponding (weighted) divergence and Laplacians. 
For simplicity, we assume weight $1$ on the vertices but any other vertex weight could be used with identical formulas. In this context, it is standard to modify a transportation network by adding two new vertices $x,y$ so that the 
$(V_0,V_1)$-flow can be rewritten as an $(x,y)$-flow with a single source $x$ and sink $y$. This is the case usually treated in the literature (see, for example, Chapter~7 in \cite{bondy-murty:08}).

\begin{itemize}
 \item[(i)] \emph{Capacity of a cut:} Consider a subset of vertices $X\subset V$ such that $x\in X$, $y\in \compl{X}$ and let $K^+:=A^+(X,\compl{X})$ and $K^-:=A^-(X,\compl{X})$ be an out-cut and in-cut of a digraph $G$, respectively. Then the capacity of the cut is defined by 
 $\operatorname{cap}(K^+) := \sum_{a \in K^+} c(a)$, i.e. with the weight interpretation it corresponds to the measure of the cut $K^+$. We obtain next an expression in terms of the out Laplacian:
\begin{eqnarray*}
\operatorname{cap}(K^+) &=& \sum_{a \in A} c(a) \cdot \1_{K^+}(a) = -\sum_{a \in A} \1_{X}(\partial_{-} a) \cdot[\1_{X}(\partial_{+} a) - \1_{X}(\partial_{-} a)] \cdot c(a) \\
                        &=&  \left\langle d^-\1_{X} ,d \1_{X}\right\rangle_{\ell_2(A,c)}
                         =    \left\langle \1_{X} ,\mathcal{L}^-\1_{X}\right\rangle_{\ell_2(V)}\;.
\end{eqnarray*}
In the previous formula the out Laplacian $\cL^-$ is considered taking the capacity function $c$ of the network as weight on the arcs. Similar formulas hold for the in-cut.

\item[(ii)] \emph{Value of a flow:}
If $\eta$ is an $(x,y)$-flow $\eta$ in $G$, we define its value $\operatorname{val}(\eta)$ as the net flow out of
the source $x$ (which coincides, due to the conservation conditions, with the net flow into $y$). Using the cut $X$ as before as well as the notions of discrete divergence in Subsection~\ref{subsec:graph_theo_operators} we obtain
\[
\begin{aligned}
    \operatorname{val}(\eta) :=& \sum_{a \in A^-_x} \eta(a) = \sum_{a \in K^+} \eta(a) - \sum_{a \in K^-} \eta(a)
    =\sum_{v \in X} \sum_{a \in A^-_v} \frac{\eta(a)}{c(a)} \cdot c(a) - \sum_{v \in X} \sum_{a \in A^+_v} \frac{\eta(a)}{c(a)} \cdot c(a)\\
    =& -\sum_{v \in X} \sum_{a \in A_v}\vec{\left(\frac{\eta}{c}\right)}_a(v)\cdot c(a) 
    = -\left\langle \1_{X} ,d^*\left(\frac{\eta}{c}\right)\right\rangle_{\ell_2(V)} 
    = -\left\langle d \1_{X} ,\left(\frac{\eta}{c}\right)\right\rangle_{\ell_2(A,c)} \;.
\end{aligned}
\]
The first equality can be verified in \cite[Proposition~7.1]{bondy-murty:08}. Note that, by handling the capacity in the same way as before, the expression for feasible flows is normalized. 
\end{itemize}

To conclude this subsection, we will demonstrate the use of the concepts previously discussed on a digraph.

\begin{example} 
We illustrate in this example the geometrical expressions obtained previously.
Consider the digraph $(G,m)$ given in Fig.~\ref{fig:exa:flow}.
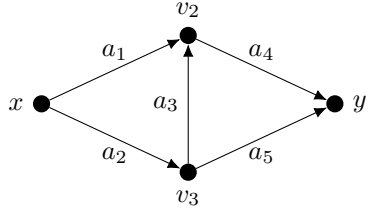
\begin{figure}[h]
		\centering
		\captionsetup{width=0.5\linewidth}
		{ 
			\begin{tikzpicture}[baseline,vertex/.style={circle,draw,fill, thick, inner sep=0pt,minimum size=2mm},scale=1.3]

				\node (1) at (-0.5,0.7) [vertex,label=left:$x$] {};
                \node (2) at (1,1.4) [vertex,label=above:$v_2$] {};
				\node (3) at (1,0)  [vertex,label=below:$v_3$] {};
				\node (4) at (2.5,0.7) [vertex,label=right:$y$] {};	
        
				\draw[-{Latex[ width=1.4mm]}](1) edge node[above] {$a_1$} (2);
				\draw[-{Latex[ width=1.4mm]}](1) edge node[below] {$a_2$} (3);
                \draw[-{Latex[ width=1.4mm]}](3) edge node[left] {$a_3$} (2);
				\draw[-{Latex[ width=1.4mm]}](2) edge node[above] {$a_4$} (4);			
				\draw[-{Latex[ width=1.4mm]}](3) edge node[below] {$a_5$} (4);
		\end{tikzpicture} } 
		\caption{Example of a digraph with one source $F=\{x\}$ and sink $S=\{y\}$.}
		\label{fig:exa:flow}
	\end{figure}
\end{example}

Using the matrix representation of the discrete derivative given by the incidence matrix $\mathcal{B}$ in Remark ~\ref{rem:ONB} we can derive the matrix for $d_0^*$. To obtain the matrix representation for operator $d_0^*$ from $\mathcal{B^*}$, we remove the rows that correspond to the source $x=v_1$ and sink $v_4=y$ vertices. For the sake of systematicity, we note it as $\mathcal{B}_0^*$.

\begin{figure}[h]
		\begin{tabular}{cc}
		    {$\mathcal{B}^* = 
				\begin{bmatrix*}
					1 & 1 & 0 & 0 & 0 \\
					\shortminus1 & 0 & \shortminus1 & 1 & 0 \\
					0 & \shortminus1 & 1 & 0 & 1 \\
                    0 & 0 & 0 & \shortminus1 & \shortminus1 \\
				\end{bmatrix*}  $,} 
            &{$\mathcal{B}_0^* = 
				\begin{bmatrix*}
					\shortminus1 & 0 & \shortminus1 & 1 & 0 \\
					0 & \shortminus1 & 1 & 0 & 1 \\
				\end{bmatrix*}  $.} 
		\end{tabular}
        \caption{Matrices $\mathcal{B}^*$, $\mathcal{B}_0^*$ for a digraph.} 
		\label{fig:exa:flow:mat}
\end{figure}

Thus, we obtain that the kernel for $d_0^*$ is given by three dimensional space
\[
\operatorname{ker}\big(d_0^*\big) =
    \left\{a \cdot \begin{bmatrix*} 0\\1\\1\\1\\0\end{bmatrix*}  + b \cdot \begin{bmatrix*} 1\\0\\0\\1\\0\end{bmatrix*} + c \cdot \begin{bmatrix*} 0\\1\\0\\0\\1\end{bmatrix*}
    \Bigm\vert a,b,c\in\R \right\}\;.
\]
Furthermore, let $X:=\{x, v_2,v_3\}\subset V$ be a subset with $x\in X \not \ni y$ and choose the capacity function $c(a_i):=i$. Then we can compute the corresponding 
\[\operatorname{cap}(K^+):= \sum_{a \in K^+} c(a) = c(a_4) + c(a_5) = 4 + 5 = 9 = \begin{bmatrix*} 1\\1\\1\\0\end{bmatrix*}^T \cdot
				\begin{bmatrix*}[r]
					3 & \shortminus1 & \shortminus2 & 0 \\
					0 & 4 & 0 & \shortminus4 \\
					0 & -3 & 8 & \shortminus5 \\
					0 & 0 & 0 & 0   \\
				\end{bmatrix*}\cdot  \begin{bmatrix*} 1\\1\\1\\0\end{bmatrix*} = \langle \1_X \,,\, {\mathcal{L^-} \1_X \rangle}.
\]

We consider an $(x,y)$-flow denoted by $\eta$, obtained by setting $a=1$, $b=1$, and $c=1$ in the previous expression for $\operatorname{ker}(d_0^*)$. Consequently, the value of $\eta$ is given by

\[
\operatorname{val}(\eta) := \sum_{a \in A^-_x} \eta(a) = \eta(a_1) + \eta(a_2) = 3 =- \begin{bmatrix*} 1\\0\\0\\0\end{bmatrix*}^T \cdot
				\begin{bmatrix*}[r]
					\shortminus1 & \shortminus2 & 0 & 0 & 0 \\
					1 & 0 & 3 & \shortminus4 & 0\\
					0 & 2 & \shortminus3 & 0 & \shortminus5 \\
					0 & 0 & 0 & 4 & 5   \\
				\end{bmatrix*}\cdot  \begin{bmatrix*} 1/1\\2/2\\1/3\\2/4\\1/5\end{bmatrix*}=
-\left\langle d \1_{x} ,\left(\frac{\eta}{c}\right)\right\rangle.\]

%
%
%

\newcommand{\etalchar}[1]{$^{#1}$}
\providecommand{\bysame}{\leavevmode\hbox to3em{\hrulefill}\thinspace}
\providecommand{\MR}{\relax\ifhmode\unskip\space\fi MR }
\providecommand{\MRhref}[2]{%
  \href{http://www.ams.org/mathscinet-getitem?mr=#1}{#2}
}
\providecommand{\href}[2]{#2}

---------------------------------------------

%

\end{document}